\documentclass[aop,noinfoline,reqno]{imsart}
\makeatletter

\RequirePackage[OT1]{fontenc}
\RequirePackage{amsthm,amsmath}
\RequirePackage[numbers]{natbib}
\RequirePackage[pdftitle={Uniformity of the late points of random walk},pdfauthor={Jason Miller and Perla Sousi},bookmarks=true,bookmarksopen=true, bookmarksopenlevel=3,unicode]{hyperref}
\RequirePackage{amssymb}
\RequirePackage{comment}
\RequirePackage{graphicx}
\RequirePackage[font=footnotesize]{caption}
\RequirePackage{subfig}
\RequirePackage{enumerate}
\RequirePackage{comment}
\RequirePackage{colonequals}
\RequirePackage{textgreek}
\RequirePackage[all]{hypcap}
\RequirePackage[normalem]{ulem}
\RequirePackage{xfrac}
\usepackage{color}
\usepackage{tikz}

\usepackage[all]{xy}
\startlocaldefs

\numberwithin{equation}{section}
\newtheorem{theorem}{Theorem}[section]
\newtheorem{lemma}[theorem]{Lemma}
\newtheorem{proposition}[theorem]{Proposition}

\theoremstyle{definition}
\newtheorem{definition}[theorem]{Definition}
\newtheorem{remark}[theorem]{Remark}

\newtheorem{claim}[theorem]{Claim}

\def\N{\mathbb{N}}
\def\Z{\mathbb{Z}}

\def\R{\mathbb{R}}

\def\AA{\mathcal{A}}
\def\SS{\mathcal{S}}
\def\F{\mathcal{F}}

\def\B{\mathcal{B}}
\def\LL{\mathcal{L}}

\def\H{\mathcal{H}}
\def\U{\mathcal{U}}
\def\V{\mathcal{V}}
\def\W{\mathcal{W}}
\renewcommand{\phi}{\varphi}
\renewcommand{\epsilon}{\varepsilon}

\allowdisplaybreaks

\renewcommand{\1}{{\text{\Large $\mathfrak 1$}}}

\newcommand{\var}{\operatorname{var}}

\renewcommand{\emptyset}{\varnothing}
\newcommand{\I}{_{\text{\sc i}}}

\newcommand{\til}{\widetilde}

\newcommand{\tmix}{t_{\mathrm{mix}}}

\newcommand{\tmax}{t_{\mathrm{hit}}}

\newcommand{\tcov}{t_{\mathrm{cov}}}
\newcommand{\tunif}{t_{\mathrm{unif}}}

\newcommand{\ball}[1]{\mathcal{B}(0,#1)}

\newcommand{\pr}[1]{\mathbb{P}\!\left(#1\right)}
\newcommand{\E}[1]{\mathbb{E}\!\left[#1\right]}
\newcommand{\estart}[2]{\mathbb{E}_{#2}\!\left[#1\right]}
\newcommand{\prstart}[2]{\mathbb{P}_{#2}\!\left(#1\right)}
\newcommand{\prcond}[3]{\mathbb{P}_{#3}\!\left(#1\;\middle\vert\;#2\right)}
\newcommand{\econd}[2]{\mathbb{E}\!\left[#1\;\middle\vert\;#2\right]}
\newcommand{\escond}[3]{\mathbb{E}_{#3}\!\left[#1\;\middle\vert\;#2\right]}
 
\newcommand{\wt}{\widetilde}

  \newcommand{\boxE}{\text{\scalebox{0.7}{$\square$}}} 
  \newcommand{\circleE}{\circ}
 
 \newcommand{\nsq}{N^{\boxE,\circleE}}
 \newcommand{\nbb}{N^{\circleE,\circleE}}
 \newcommand{\nsqsq}{N^{\boxE,\boxE}}
 \newcommand{\trs}{T^{\boxE,\circleE}}
 \newcommand{\trb}{T^{\circleE,\circleE}}

\newcommand{\ol}{\overline}
\newcommand{\ul}{\underline}

\newcommand{\norm}[1]{\left\| #1 \right\|}

\newcommand{\tn}{|\kern-.1em|\kern-0.1em|}

\newcommand{\crw}{C_d}

\endlocaldefs

\begin{document}

\begin{frontmatter}
\title{Uniformity~of~the~late~points~of random~walk~on~$\Z_{\lowercase{n}}^{\lowercase{d}}$~for~$\lowercase{d} \geq 3$}
\runtitle{Uniformity of the late points of random walk}
\begin{aug}
\author{Jason Miller} and
\author{Perla Sousi}

\runauthor{Jason Miller and Perla Sousi}

\address{
Jason Miller\\
Massachusetts Institute of Technology\\
Cambridge, MA
}

\address{
Perla Sousi\\
University of Cambridge\\
Cambridge, UK
}

\affiliation{Massachusetts Institute of Technology 
and University of Cambridge}
\end{aug}

\begin{abstract}
Suppose that $X$ is a simple random walk on $\Z_n^d$ for $d \geq 3$ and, for each $t$, we let $\U(t)$ consist of those $x \in \Z_n^d$ which have not been visited by $X$ by time $t$.  Let $\tcov$ be the expected amount of time that it takes for $X$ to visit every site of $\Z_n^d$.  We show that there exists $0 < \alpha_0(d) \leq \alpha_1(d) < 1$ and a time $t_* = \tcov(1+o(1))$ as $n \to \infty$ such that the following is true.  For $\alpha > \alpha_1(d)$ (resp.\ $\alpha < \alpha_0(d)$), the total variation distance between the law of $\U(\alpha t_*)$ and the law of i.i.d.\ Bernoulli random variables indexed by $\Z_n^d$ with success probability~$n^{-\alpha d}$ tends to~$0$ (resp.\ $1$) as $n \to \infty$.  Let $\tau_\alpha$ be the first time $t$ that $|\U(t)| = n^{d-\alpha d}$.  We also show that the total variation distance between the law of $\U(\tau_\alpha)$ and the law of a uniformly chosen set from $\Z_n^d$ with size $n^{d-\alpha d}$ tends to $0$ (resp.\ $1$) for $\alpha > \alpha_1(d)$ (resp.\ $\alpha < \alpha_0(d)$) as $n \to \infty$.
\end{abstract}

\thispagestyle{empty}

\setattribute{keyword}{AMS}{AMS 2010 subject classifications:}


\date{\today}

\end{frontmatter}

\maketitle

\section{Introduction}
\label{sec:intro}

Suppose that $X$ is a simple random walk on $\Z_n^d$ for $d \geq 3$ started from the stationary distribution.  
For each~$x \in \Z_n^d$, we let 
\[\tau_x = \min\{t \geq 0 : X(t) = x\}\]
be the first time that $X$ visits $x$.  For~$t\geq 0$ we define the process
$(Q_x(t))$ and the set $\U(t)$ respectively by
\[Q_x(t) = \1(\tau_x>t) \quad \text{for}\quad x\in 
\Z_n^d \quad \text{and} \quad
\U(t) = \{ x \in \Z_n^d : Q_x(t) = 1\}.\]
The purpose of the present work is to study the law of the set $\U(t)$ for different values of~$t$. 
The correlation structure of $(Q_x(t))$ was analyzed in the physics literature by Brummelhuis and Hilhorst \cite{BH}.  They show that the probability that any two given points~$x,y \in \Z_n^d$ which are far from each other are not visited by time~$t$ is asymptotically the same as in the case in which the points are independent, i.e., $\pr{Q_x(t) =1, Q_y(t) = 1} \sim \pr{Q_x(t) = 1}\pr{Q_y(t) = 1}$ as~$t,n\to \infty$ at a certain rate.  This leads them to assert that $\U(t)$ is ``statistically uniformly distributed at large distances'' \cite[Section~4]{BH}.  In this article, we study in what sense the \emph{entire} joint law of $(Q_x(t))$ is uniformly distributed for ``large'' times~$t$ rather than focus on its finite dimensional distributions.

In order to state our results and put them into better context with the existing literature, we first introduce the following parameters for $X$.  The \emph{maximal hitting time} ($\tmax$) and \emph{cover time} ($\tcov$) are respectively given by
\[\tmax = \max_{x,y} \estart{\tau_y}{x} \quad\text{and}\quad \tcov = \max_x \estart{\max_y \tau_y}{x}.\]
The times $\tmax,\tcov$ are related in that $\tcov = \tmax \log (n^d) (1+o(1))$ (see \cite{MATT_cover} as well as \cite[Chapter~11]{LPW}, in particular \cite[Exercise~11.4]{LPW}).  The rate at which the $o(1)$ term tends to $0$ will be important for technical reasons so in some cases we will describe times in terms of $\tmax$ or other ways rather than directly in terms of $\tcov$.  For measures~$\mu$ and~$\nu$, we recall that the \emph{total variation distance} is given by
\[ \| \mu-\nu\|_{\rm TV} = \sup_A |\mu(A) - \nu(A)|\]
where the supremum is taken over all measurable subsets $A$.

We will analyze the structure of $\U(t)$ at times of the form $\alpha \tcov$ for $\alpha > 0$.  We mention here three important regimes of $\alpha$.  The first is when $\alpha > 1$.  It is a consequence of work by Aldous \cite{ALD_thresh} that for any $\alpha > 1$ and $t > \alpha \tcov$ we have $\U(t) = \emptyset$ with high probability.  The case that $\alpha = 1$ was studied by Belius~\cite{BEL} using random interlacements \cite{SZNIT_interlacements} and later by Imbuzeiro-Oliveira and Prata~\cite{ImbuzPrata, Prata_thesis} using hitting time estimates \cite{IMBUZ_hit}.  The main focus of~\cite{BEL} is to obtain the Gumbel fluctuations of the cover time of $\Z_n^d$ and as a consequence of his analysis he shows in~\cite[Corollary~2.4]{BEL} that the set of uncovered points at time~$t_\beta = \tmax (\log(n^d) + \beta)$ for $\beta\in \R$ suitably rescaled converges to a Poisson point process on $(\R/\Z)^d$ of intensity~$e^{-\beta}\lambda$ where $\lambda$ denotes Lebesgue measure on $(\R/\Z)^d$. This was improved upon in~\cite{ImbuzPrata, Prata_thesis}, where it is shown that the Gumbel fluctuations for the cover time hold for more general graphs. Moreover they show that the total variation distance between the law of~$\U(t_\beta)$ and that of a random subset of~$\Z_n^d$ where points are included independently with probability~$e^{-\beta}n^{-\alpha d}$ tends to~$0$ as $n\to \infty$. The regime of times considered in~\cite{BEL, ImbuzPrata, Prata_thesis} is special because $|\U(t_\beta)|$ is tight as $n \to \infty$ for any fixed $\beta \in \R$.  Additionally, the law of the evolution of $\U(t_\beta)$ as $\beta$ varies is also described in~\cite{ImbuzPrata, Prata_thesis}.

The final regime of times is when $\alpha\in (0,1)$.  In contrast to the cases described above, for such choices of $\alpha$ the size of $|\U(t)|$ grows with $n$.  In particular, it is shown in the proof of \cite[Theorem~4.1]{PR_lamp} that it follows from \cite{ALD_thresh} that $|\U(t)| = n^{d-\alpha d + o(1)}$ with high probability as $n \to \infty$.  The combinatorial method of \cite{ImbuzPrata, Prata_thesis} does not extend directly to this regime of times because the number of possible sets one is led to consider is simply too large.  The following alternative ``uniformity'' statement for $\U(t)$ was proved in \cite{MP_unif}.  If $\alpha \in (\tfrac{1}{2},1)$ (resp.\ $\alpha \in (0,\tfrac{1}{2})$) then $\U(t)$ is (resp.\ is not) ``uniformly random'' in the following sense.  Suppose that $\V \subseteq \Z_n^d$ is chosen independently of $X$ where each $x \in \Z_n^d$ is included in $\V$ independently with probability $\tfrac{1}{2}$.  Then the total variation distance between the laws of $\V \setminus \U(t)$ and $\V$ tends to $0$ (resp.\ $1$) as $n \to \infty$ for $\alpha \in (\tfrac{1}{2},1)$ (resp.\ $(0,\tfrac{1}{2})$).  That is, for $\alpha \in (\tfrac{1}{2},1)$, $\U(t)$ in a certain sense does not possess any sort of systematic geometric structure that would make it possible to determine from $\V \setminus \U(t)$ the location of the points in $\U(t)$.  The threshold $\alpha=\tfrac{1}{2}$ is important because $|\V \setminus \U(t)| = n^{d-\alpha d + o(1)}$ for $\alpha \in (0,\tfrac{1}{2})$ while $|\V| = n^{d/2+o(1)}$ by the central limit theorem, so in this case the two sets can be distinguished for elementary reasons.  We remark in passing that a similar problem for ``thin'' 3D torii is considered in \cite{DDMP} and the $d=2$ version of this problem is solved in \cite{PR_lamp} using results from \cite{DPRZ_late}.

In contrast to \cite{MP_unif}, in this work we are going to study the asymptotic law of~$\U(t)$ itself in the sense of~\cite{ImbuzPrata, Prata_thesis} in the regime of times with $\alpha \in (0,1)$ \emph{without adding the extra noise}.  It will be rather important for us to choose the time $t$ at which we consider $\U(t)$ very precisely since we will later need a very accurate estimate of~$\pr{\tau_x > t}$.  In the theorem statement which follows, $t_*$ indicates a time which we will define later in the article (equation~\eqref{eq:deftstar}) and it satisfies
\[ t_*=\tmax\log (n^d)(1+o(1)) = \tcov(1+o(1)) \ \text{ as } \ n\to\infty. \] 
For any $\alpha > 0$ we denote by $\nu_{\alpha, n}$ the law of $(Z_x)$, where the $Z_x$ are i.i.d.\ Bernoulli random variables indexed by $\Z_n^d$ with success probability $n^{-\alpha d}$.  We will write $\LL(\cdot)$ to indicate the law of a random variable.  Our first main result is the following.

\begin{theorem}
\label{thm:uniform}
For each $d\geq 3$ there exist $\alpha_0(d),\alpha_1(d)\in (0,1)$ with $\alpha_0(d) \leq \alpha_1(d)$ such that for all $\alpha \in (\alpha_1(d),\infty)$ we have
\begin{align}
\|\LL(\U(\alpha t_*)) - \nu_{\alpha,n}\|_{\rm{TV}} &= o(1) \ \text{ as } n\to \infty \label{eq:unif_a0}
\intertext{and for all $\alpha \in (0,\alpha_0(d))$ we have}
\|\LL(\U(\alpha t_*))  - \nu_{\alpha,n}\|_{\rm{TV}} &=1 - o(1) \ \text{ as } n\to \infty. \label{eq:unif_a1}
\end{align}
\end{theorem}

In analogy with \cite{DPRZ_late}, we refer to the points in $\U(\alpha t_*)$ as ``$\alpha$-late'' for $X$. The reason for the terminology ``late'' is that the amount of time required by $X$ to hit them is much larger than the maximal hitting time.  Our definition of $\alpha$-late is slightly different than that given in~\cite{DPRZ_late} because we use $t_*$ instead of~$\tcov$.

Let $p_d$ be the probability that a simple random walk in $\Z^d$ starting from $0$ returns to $0$ before escaping to $\infty$.
The values of $\alpha_0(d)$ and $\alpha_1(d)$ from Theorem~\ref{thm:uniform} are explicitly given by
\begin{align*}
\alpha_0(d) = \frac{1+p_d}{2} \quad\text{and} \quad
\alpha_1(d) = \frac{(\kappa-2)d + d\kappa}{(\kappa-2) (d+1)+ d\kappa} \quad\text{where}\quad \kappa = d \wedge 6.
\end{align*}
The threshold $\alpha_0(d)$ is special because, as we show in Sections~\ref{sec:separated} and~\ref{sec:totalvar}, $\U(\alpha t_*)$ with high probability has neighbouring points for $\alpha \in (0,\alpha_0(d))$ but does not for $\alpha > \alpha_0(d)$.  In fact, for every $\alpha > \alpha_0(d)$ the distance between any pair of distinct points in $\U(\alpha t_*)$ is at least $n^{p_d}$ with high probability.  That is, the minimal distance between distinct points in $\U(\alpha t_*)$ jumps from $0$ to being larger than $n^{p_d}$ as $\alpha$ crosses the threshold $\alpha_0(d)$ with high probability.  We emphasize that $\alpha_0(d) > \tfrac{1}{2}$ for all $d \geq 3$ and $\alpha_0(d) \to \tfrac{1}{2}$ as $d \to \infty$.  The value $\tfrac{1}{2}$ is significant due to the connection between this work and \cite{MP_unif} described above.

Theorem~\ref{thm:uniform} describes the asymptotic behavior of the law of $\U(t)$ at a deterministic time $t$ of a specific form.  In our second main result, we describe the asymptotic behavior of $\U(\tau)$ where $\tau$ is the first time $t$ that $\U(t)$ contains a certain number of points.  More specifically, for each $\alpha > 0$, we let
\[
\tau_\alpha=\inf\{t\geq 0: |\U(t)|=n^{d-\alpha d}\}.
\]
We also let $\W_\alpha$ be a subset of $\Z_n^d$ picked uniformly at random among all subsets of $\Z_n^d$ containing exactly $n^{d-\alpha d}$ points. Then we have the following:

\begin{theorem}
\label{thm:exact}
Suppose that $d\geq 3$ and that $\alpha_0(d),\alpha_1(d) \in (0,1)$ are as in Theorem~\ref{thm:uniform}.  For all $\alpha \in (\alpha_1(d),\infty)$, we have
\begin{align}
\|\LL(\U(\tau_\alpha)) - \LL(\W_\alpha) \|_{\rm{TV}} &= o(1) \quad \text{as} \quad n\to \infty \label{eq:exact_a0}
\intertext{and for all $\alpha \in (0,\alpha_0(d))$ we have}
\|\LL(\U(\tau_\alpha)) - \LL(\W_\alpha) \|_{\rm{TV}} &= 1 - o(1) \quad \text{as} \quad n\to \infty. \label{eq:exact_a1}
\end{align}
\end{theorem}

We will derive Theorem~\ref{thm:exact} from Theorem~\ref{thm:uniform} using an estimate which gives that the first hitting distribution of $X$ on $A \subseteq \Z_n^d$, where $A$ is a set of points which is ``well-separated,"  is closely approximated by the uniform distribution on $A$.

A number of questions naturally arise from this work (exact values where the transitions from non-uniformity to uniformity occur, existence of a phase transition, behaviour for $\alpha \in (0,\alpha_0(d))$, other graphs, etc...) which we state more carefully in Section~\ref{sec:questions}.

\subsection{Relation to other work}
\label{subsec:previous work}

The structure of $\U(\alpha \tcov)$ for $d=2$ was also studied in the physics literature by \cite{BH} and later in the mathematics literature by \cite{DPRZ_late}.  In contrast to the case that $d \geq 3$, $\U(\alpha \tcov)$ for $d=2$ is not uniform for any $\alpha \in (0,1)$.  In particular, the last visited set tends to organize itself into clusters which are of diameter up to $n^\beta$ where $\beta=\beta(\alpha) > 0$ for any $\alpha \in (0,1)$.  The reason for the difference is that random walk for $d=2$ is recurrent which leads to longer range correlations while for $d \geq 3$ it is transient.  Thus the process of coverage in the two regimes is very different.  The work \cite{DPRZ_late} is part of a larger series which also includes \cite{DPRZ_thick, DPRZ_bm_manifold, DPRZ_cov} and the proofs of Theorems~\ref{thm:uniform} and~\ref{thm:exact} employ several techniques which are present in the articles of this series.

\subsection{Notation and assumptions}
\label{subsec:notation_and_assumptions}

Throughout this article, we shall always assume that $d \geq 3$ unless explicitly stated otherwise.  For functions $f,g$ we will write $f(n) \lesssim g(n)$ if there exists a constant $c > 0$ such that $f(n) \leq c g(n)$ for all $n$.  We write $f(n) \gtrsim g(n)$ if $g(n) \lesssim f(n)$.  Finally, we write $f(n) \asymp g(n)$ if both $f(n) \lesssim g(n)$ and $f(n) \gtrsim g(n)$.  Many of the proofs will involve a number of different constants which we will often indicate simply by $c$.
We write $\mathbb{P}$ without the subscript $\pi$ to indicate the law of a simple random walk in~$\Z_n^d$ started from stationarity. We will also write $\mathbb{P}_x$ to indicate the law of the random walk when started from~$x$.  We denote by $\mathbb{E}$ and $\mathbb{E}_x$ the corresponding expectations.

\subsection{Strategy}
\label{subsec:strategy}

The proofs of Theorems~\ref{thm:uniform} and~\ref{thm:exact} require many different estimates.  We now provide an overview of the different steps and how they fit together.  Throughout, we assume that we have fixed some value of $\alpha \in (0,1)$ and $d \geq 3$.

\smallskip

\begin{figure}[ht!]
\begin{center}
\includegraphics[scale=0.85,page=1]{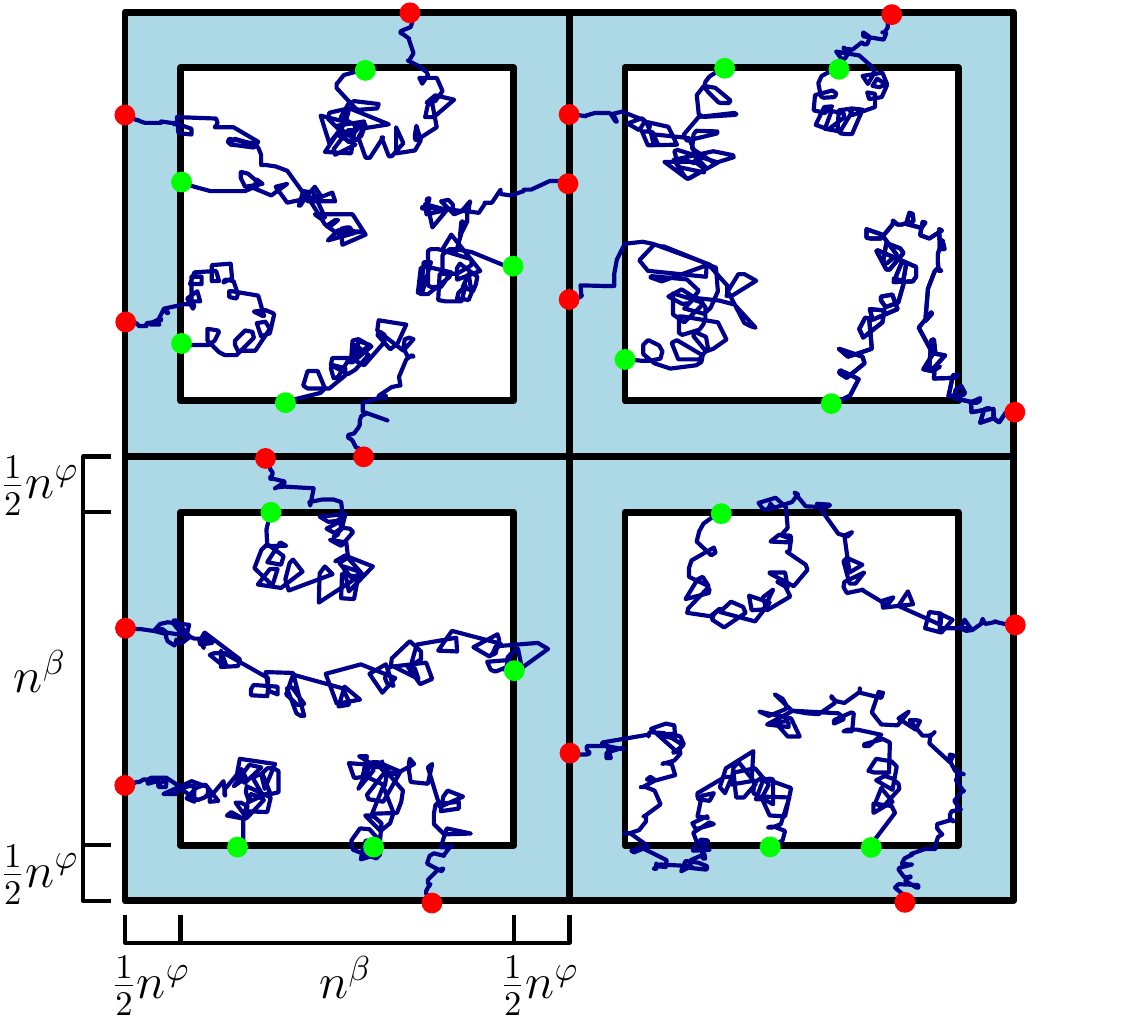}
\end{center}
\caption{\label{fig:spatial_decomposition} Four boxes of side length $n^\beta + n^\phi$ in the spatial decomposition of $\Z_n^d$ used in the proofs of Theorems~\ref{thm:uniform} and~\ref{thm:exact} are illustrated above.  The white inner boxes represent the concentric boxes of side length $n^\beta$.  We denote by $\SS_\beta$ the collection of all such white boxes and for each $S \in \SS_\beta$ we let $\ol{S}$ (resp.\ $\ul{S}$) be the concentric box of side length $n^\beta + n^\phi$ (resp.\ $n^\beta - n^\phi$) which contains it (resp.\ contained in it).  For $\alpha > (d+\phi)/(d+1)$, with high probability there are no unvisited points in $\AA = \Z_n^d \setminus \cup_{S \in \SS_\beta} \ul{S}$.  In the setting of the modified version of the problem described in Step 1 in Section~\ref{subsec:strategy}, conditional on the entrance and exit points of the excursions that $X$ makes between the boundaries of the boxes in $S \in \SS_\beta$ and $\ol{S}$, the sets of unvisited points in the different $\ul{S}$ for $S \in \SS_\beta$ are independent.  Shown are a few such excursions in dark blue.  The entrance (resp.\ exit) points are indicated by green (resp.\ red) disks.  These are just a caricature; in the proofs~$\phi$ is taken to be much smaller than $\beta$ so most of the excursions are in fact very short and end very close to where they start.}
\end{figure}

\noindent\textbf{Spatial decomposition:}
We fix two small parameters $\epsilon, \phi \in (0,1)$ and let $\beta = \alpha-\epsilon$.  We then partition $\Z_n^d$ into disjoint boxes of side length $n^\beta+n^\phi$ and consider in each such box concentric sub-boxes of side lengths $n^\beta-n^\phi$ and $n^\beta$ (see Figure~\ref{fig:spatial_decomposition}).  We let $\SS_\beta$ denote the collection of the latter type of concentric boxes and for each $S \in \SS_\beta$ we let $\ol{S}$ (resp.\ $\ul{S}$) be the box with side length $n^\beta+n^\phi$ (resp.\ $n^\beta-n^\phi$) which contains it (resp.\ is contained in it).  We also let $\AA = \Z_n^d \setminus \cup_{S \in \SS_\beta} \ul{S}$ be the region between the outside and inside boxes.  Note that $|\AA| \asymp n^{d-d\beta} \times n^{(d-1)\beta + \phi} = n^{d-\beta+\phi}$.  The probability that a given point is not visited at time $\alpha t_*$ is $n^{-\alpha d(1+o(1))}$; this follows from the proof of \cite[Theorem~4.1]{PR_lamp} using \cite{ALD_thresh} as mentioned earlier and the vertex transitivity of $\Z_n^d$ (we will also give a more precise version of this result which is specific to $\Z_n^d$). Consequently, for $\alpha > (d+\phi)/(d+1)$ we can choose $\epsilon > 0$ small enough so that we have $\AA \cap \U(\alpha t_*) = \emptyset$ with high probability.  Therefore it suffices to prove the uniformity of the last visited points which are contained in $\cup_{S \in \SS_\beta} \ul{S}$.  This leads us to consider the following modified version of the problem.  We let $\wt{\U}(\alpha t_*)$ consist of those points in each box $\ul{S}$ for $S \in \SS_\beta$ which have not been visited by the first time that the number of excursions made by $X$ from $\partial S$ to $\partial \ol{S}$ by time~$\alpha t_*$ exceeds the typical number~$E$.  We show that we have sufficiently good concentration for the number of such excursions up to a given time so that $\U(\alpha t_*) = \wt{\U}(\alpha t_*)$ with high probability.  We then prove the uniformity of $\wt{\U}(\alpha t_*)$.  This modified problem is useful to consider because the random variables $(\wt{\U}(\alpha t_*) \cap \ul{S})_{S \in \SS_\beta}$ are independent conditional on the $\sigma$-algebra $\F$ generated by the entrance and exit points of these excursions.  Thus to bound the total variation distance between $\LL(\wt{\U}(\alpha t_*))$ and $\nu_{\alpha,n}$ it suffices to bound the expectation of the sum of the total variation distances between the conditional laws of the last visited set in each $\ul{S}$ for $S \in \SS_\beta$ given $\F$ and a random subset of $\ul{S}$ where points are included independently with probability~$n^{-\alpha d}$ (explained below).  

\smallskip

\begin{figure}[ht!]
\begin{center}
\includegraphics[scale=0.85]{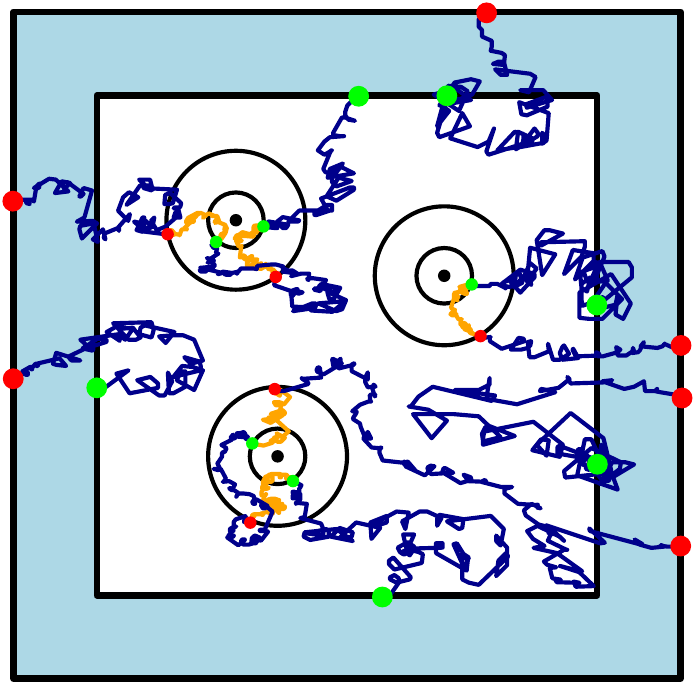}
\end{center}
\caption{\label{fig:box_separated} (Continuation of Figure~\ref{fig:spatial_decomposition}) A single box $\ol{S}$ of side length $n^\beta + n^\phi$ is shown along with the corresponding concentric box $S \in \SS_\beta$ with side length $n^\beta$.  Inside $S$, three points are shown and around each point we have placed two concentric balls.  Conditional on the number and entrance and exit points of the excursions  (illustrated in orange above) that $X$ makes across each of these spherical annuli during a given number of excursions across $\ol{S} \setminus S$, the events that each of the points are hit is independent.}
\end{figure}

\noindent{\bf Uniformity in each box:}
Our strategy for proving the uniformity of $\wt{\U}(\alpha t_*) \cap \ul{S}$ for a given $S \in \SS_\beta$ is based on the same high level idea used  in \cite{ImbuzPrata, Prata_thesis} (inclusion-exclusion and the Bonferroni inequalities) though the implementation is different.  The first step is to show that for each $\epsilon > 0$ there exists $M <\infty$ so that with high probability $\max_{S \in \SS_\beta} |\wt{\U}(\alpha t_*) \cap \ul{S}| \leq M$.  We also show that with high probability $\wt{\U}(\alpha t_*) \cap \ul{S}$ is ``well-separated'' in the sense that for some choice of $\gamma > 0$, the distance between any two distinct points $x,y \in \wt{\U}(\alpha t_*) \cap \ul{S}$ is at least $n^\gamma$. Thus to bound the total variation distance, we can restrict our attention to finite, well-separated sets.  To complete the proof, we need very precise hitting estimates in order to determine the probability that any given such set $S \subseteq \ul{S}$ for $S \in \SS_\beta$ is not visited by~$X$ during its first~$E$ excursions from $\partial S$ to $\partial \ol{S}$.  This needs to be sufficiently precise so that we can sum the error over all possible well-separated subsets of~$\ul{S}$ of size~$M$ and then sum that error over all of the boxes in $\SS_\beta$.  To accomplish this, we put spherical annuli (see Figure~\ref{fig:box_separated}) around each of the points in $S$ with in-radius $n^{2\phi/\kappa}$ for $\kappa = d \wedge 6$ and out-radius $n^\phi$ (the sizes and the value of $\phi$ are chosen to optimize several error terms).  Conditional on the number of excursions $N$ that $X$ makes across each such spherical annulus and their entrance and exit points as well as the corresponding data for the first $E$ excursions from $\partial S$ to $\partial \ol{S}$, the probability that each point is hit is independent.  Another concentration estimate implies that $N$ is with high probability very close to the typical number made by $X$ by time $\alpha t_*$, so we can replace it with this deterministic value.  Moreover, estimates for discrete harmonic functions \cite{LawLim} give us that the probability that a given excursion hits a point does not depend strongly on its entrance and exit points.  Putting everything together finishes this step.

\smallskip

\noindent{\bf Non-uniformity for small $\alpha$:}
The next step in the proof of Theorem~\ref{thm:uniform} is to establish the existence of $\alpha_0(d)$, i.e., that for small values of $\alpha$ the total variation distance between the law of $\U(\alpha t_*)$ and~$(Z_x)$ tends to $1$ as $n \to \infty$.  The idea is to show that for sufficiently small values of $\alpha$, the number of unvisited points which have an unvisited neighbour is much larger for $\U(\alpha t_*)$ than for $(Z_x)$.

\smallskip

\noindent{\bf Uniformity of $\U(\tau_\alpha)$:}
The final step is to deduce Theorem~\ref{thm:exact} from Theorem~\ref{thm:uniform}.  The main idea is to show that for any well-separated collection of points $A$, the first exit distribution of $X$ from~$\Z_n^d \setminus A$ is close to the uniform measure on $A$ provided $X$ starts sufficiently far from $A$.   By Theorem~\ref{thm:uniform}, if we fix $\epsilon > 0$ very small and run $X$ until time $(\alpha-\epsilon) t_*$ then we know that $\U((\alpha-\epsilon)t_*)$ is close in law to a random subset of~$\Z_n^d$ where points are included independently with probability $n^{-(\alpha-\epsilon)d}$.  Using the aforementioned estimate, for $t \geq (\alpha-\epsilon)t_*$ the random walk~$X$  decimates $\U(t)$ by removing points one by one uniformly at random.  The estimate for the uniformity of the first exit distribution is good enough that we can sum the error over the $\asymp n^{d-(\alpha-\epsilon)d}$ points necessary to remove until the last visited set has size exactly $n^{d-\alpha d}$ provided we choose $\epsilon > 0$ small enough.

\subsection{Outline}
\label{subsec:outline}

The remainder of this article is structured as follows.  In Section~\ref{sec:excursions}, we establish several concentration estimates for the number of excursions that~$X$ makes across annuli of different widths.  Next, in Section~\ref{sec:hitting} we establish a number of estimates related to the probability that an excursion of~$X$ hits points.  The purpose of Section~\ref{sec:separated} is to prove some preliminary results on the structure of the last visited set.  In particular, we show that the points which have not been visited by time~$\alpha t_*$
for large enough values of $\alpha$ are typically far from each other.  In Section~\ref{sec:totalvar}, we complete the proof of Theorem~\ref{thm:uniform} and in Section~\ref{sec:exact} we derive Theorem~\ref{thm:exact} from Theorem~\ref{thm:uniform}.  Finally, in Section~\ref{sec:questions} we list a number of open questions which naturally arise from this work.

\section{Excursions}\label{sec:excursions}

Let $r<R$. We write $\SS(x,r)$ for the box centered at $x$ of side length~$r$ and $\B(x,r)$ for the closed Euclidean ball centered at $x$ of radius $r$.  For sets $E(x,r) = \B(x,r) \text{ or } \SS(x,r)$ and $F(x,R) = \B(x,R) \text{ or } \B(x,R)$ with $E(x,r)\subseteq F(x,R)$ we define a sequence of stopping times
\begin{align*}
\tau_0=\inf\{t\geq 0: X(t) \in \partial E(x,r)\}, \\
\sigma_0 = \inf\{t\geq \tau_0: X(t)\notin F(x,R)\}
\end{align*}
and inductively we set 
\begin{align*}
\tau_{k+1} = \inf\{t\geq \sigma_{k}: X(t) \in \partial E(x,r)\} \\
\sigma_{k+1} = \inf\{ t\geq \tau_{k+1}: X(t) \notin F(x,R)\},
\end{align*}
where $E$ and $F$ will be understood from the context.

\begin{definition}\rm{
We call a path of the random walk trajectory an {\it{excursion}} if it starts from $F(x,R)$ and it comes back to $\partial F(x,R)$ after hitting $E(x,r)$.
}\end{definition}

We now define $\nsq_x(r,R,t)$ to be the total number of excursions across the annulus $\B(x,R)\setminus \SS(x,r)$ before time~$t$. More formally for $E(x,r)=\SS(x,r)$ and $F(x,R) = \B(x,R)$ we let
\[
\nsq_x(r,R,t) = \min\left\{k\geq 0: \sum_{i=1}^{k}(\sigma_i-\sigma_{i-1}) +(\sigma_0 -\tau_0)\geq  t\right\}.
\]
Similarly we define $\nsqsq_x(r,R,t)$ for the number of excursions in the annulus $\SS(x,R)\setminus \SS(x,r)$ before time $t$ and finally
$\nbb_x(r,R,t)$ for the excursions across $\B(x,R)\setminus\B(x,r)$ before time~$t$.

\begin{lemma}\label{lem:coupling}
Let $R \geq 10r$ and let $Y_j$ be the exit point of the $j$-th excursion across $\ball{R}\setminus\ball{r}$ or across $\ball{R}\setminus\SS(0,r)$. Then $(Y_j)_j$ is a finite state space Markov chain with a stationary distribution $\til{\pi}$ and mixing time of order $1$, i.e.\ there exists $k_0<\infty$ such that $\tmix = k_0$. 
Fix $\psi>0$ and let $N=k_0n^\psi$. Then there exists a positive constant $c$ such that for all~$m$ we have
\[
\|\LL(Y_N,\ldots, Y_{mN}) -\til{\pi}^{\otimes m}\|_{\rm{TV}} \lesssim me^{-cN}.
\]
\end{lemma}
\begin{proof}
See Appendix~\ref{app:elementary}.
\end{proof}

\begin{definition}\label{def:trR}\rm{
For $R\geq 10r$ we let 
\[
\trs_{r,R} = \estart{\sigma_1 - \sigma_0}{\til{\pi}},
\]
i.e.\ $\trs_{r,R}$ is the expected length of the excursion
when the walk is started on $\partial\ball{R}$ according to the stationary distribution $\til{\pi}$ of the exit points of the excursions across the annulus $\B(0,R)\setminus\SS(0,r)$ as given in Lemma~\ref{lem:coupling}. 
We define~$\trb_{r,R}$ similarly except that the excursions are across the annulus $\B(0,R)\setminus\B(0,r)$. 
}
\end{definition}

\begin{lemma}\label{lem:n1}
For each $\psi\in (0,1/2)$ there exists $n_0\geq 1$ and a positive constant $c$ such that for all $n\geq n_0$ the following is true. Suppose that $n/4\geq R\geq 10r$ and $t\asymp n^d\log n$. Then for all~$\delta>0$ such that~$\delta r^{d-2} n^{-\psi-1/2} \leq 1$ and~$\delta n^{\psi} \leq 1$ we have that for all $x$
\[
\pr{\nsq_x(r,R,t) \notin [A,A']} \lesssim  n^\psi e^{-c\delta^2 r^{d-2}/n^\psi} + e^{-cn^\psi},
\]
where $A=t/((1+\delta) \trs_{r,R})$ and $A'=t/((1-\delta) \trs_{r,R})$.
\end{lemma}

\begin{remark}\label{rem:samestatement}\rm{
We note that Lemma~\ref{lem:n1} holds when we replace $\nsq,\trs$ by $\nbb,\trb$ respectively. The proof is identical to the one given below.
}
\end{remark}

\begin{proof}[Proof of Lemma~\ref{lem:n1}]

To simplify notation throughout the proof we simply write $N_1=\nsq_x(r,R,t)$ and $T_{r,R}=\trs_{r,R}$.
In order to avoid carrying too many constants, we will prove the result for~$t=n^d\log n$. The proof for~$t\asymp n^d\log n$ is exactly the same. 
Let $N=k_0n^\psi$, where $k_0$ is the mixing time  of the exit point chain as in Lemma~\ref{lem:coupling}.

Note that $A, A'\asymp r^{d-2}\log n$ by Lemma~\ref{lem:trRlem}. In the following proof we will write either~$A$, $A'$ or the expression above depending on whichever is more convenient.

We first show that 
\begin{align}\label{eq:lowergoal}
\pr{N_1<A} \lesssim Ne^{-c\delta^2  A/N} +  e^{-cN}.
\end{align}
Let $V_i=\sigma_i-\sigma_{i-1}$ for all $i\geq 1$.
By the definition of~$N_1$ we get
\[
\pr{N_1<A}= \pr{\sum_{i=1}^{A}V_i  + (\sigma_0 - \tau_0)\geq t}.
\]
It is easy to see that there exists a positive constant $c$ such that 
\begin{align}\label{eq:markov}
\pr{\sigma_0-\tau_0 \geq  n^2\cdot \sqrt{n}} \leq e^{-c\sqrt{n}}.
\end{align}
Indeed, $\sigma_0-\tau_0$ is the time it takes for the random walk to exit the ball $\B(x,R)$ when started from~$\partial\B(x,r)$. 
Since $R\leq n/4$ and the total variation mixing time~$\tmix\asymp n^2$(see for instance~\cite[Theorem~5.5 and Example~7.4.1]{LPW}), the probability that this time is $\gtrsim n^2$ is $\leq 1/2$, so iterating the Markov property proves~\eqref{eq:markov}.
Since $t =n^d\log n$ we obtain
\begin{align*}
\pr{N_1<A}&\leq \pr{\sum_{i=1}^{A}V_i>t\left(1-\frac{1}{n^{d-5/2}\log n}\right)} + e^{-c\sqrt{n}} 
\\ &\leq \pr{\sum_{i=1}^{A}V_i>t\left(1-\frac{1}{n^{d-5/2}\log n}\right)} + e^{-cN}, 
\end{align*}
since $\psi<1/2$.
It thus suffices to show for some positive constant $c$ we have that 
\begin{align}\label{eq:suffice}
\pr{\sum_{i=1}^{A}V_i>t\left(1-\frac{1}{n^{d-5/2}\log n}\right)} \lesssim Ne^{-c\delta^2 A/N}+ e^{-c N}.
\end{align}
In order to prove~\eqref{eq:suffice} we will establish the concentration of the sequence $(V_i)_i$. 
The idea is that if we allow enough time so that the corresponding exit point chain of Lemma~\ref{lem:coupling} mixes, then the times $(V_i)_i$ are essentially i.i.d.\ so we can apply a concentration inequality for i.i.d.\ random variables. 

Let $t'=t\left(1-\frac{1}{n^{d-5/2}\log n} - \frac{c_1n^{2\psi}}{r^{d-2}\log n}\right)$ for a positive constant $c_1$. We will set the value of $c_1$ later in the proof.  Observe that
\begin{align}
\label{eq:toohuge}
\begin{split}
      &  \pr{\sum_{i=1}^{A}V_i>t\left(1-\frac{1}{n^{d-5/2}\log n}\right)}\\
\leq& \pr{\sum_{i=1}^{N-1}V_i\geq \frac{c_1n^{d+2\psi}}{r^{d-2}}}
+ \pr{\sum_{i=N}^{A}V_i>t'}.
\end{split}
\end{align}
Since by Lemma~\ref{lem:trRlem} we have $\E{V_i}\asymp n^d/r^{d-2}$ uniformly over all starting points in $\partial \B(x,R)$, by the Markov property we have by possibly decreasing the value of $c>0$
\[
\max_x \prstart{V_i\geq \frac{c_1 n^{d+\psi}}{r^{d-2}}}{x} \lesssim e^{-cN}.
\]
Hence using the union bound we get that 
\begin{align}\label{eq:firstterm}
\pr{\sum_{i=1}^{N-1}V_i\geq\frac{c_1n^{d+2\psi}}{r^{d-2}}} \lesssim N e^{-cN}.
\end{align}
By decreasing the value of $c > 0$, the above is in turn $\lesssim e^{-cN}$.  It remains to bound the second term appearing on the right hand side of~\eqref{eq:toohuge}.
By applying a union bound and the strong Markov property we get
\begin{align}\label{eq:vit'}
\pr{\sum_{i=N}^{A}V_i>t'}  \leq N\max_x\prstart{\sum_{i=1}^{A/N} V_{iN}>\frac{t'}{N}}{x}
\end{align}
Let $(Z_i)$ be i.i.d.\ distributed according to $\til{\pi}$ and $(W_i)$ be i.i.d.\ excursion lengths across the annulus $\B(x,R) \setminus \B(x,r)$ when the starting point is $Z_i$. Let $(Y_i)$ be the exit points of the excursions of the random walk. Then we couple $(V_i)_{i\geq N}$ with $(W_i)_{i\geq N}$ as follows: by Lemma~\ref{lem:coupling} the optimal coupling for $Y=(Y_N,Y_{2N},\ldots,Y_{A})$ and $Z=(Z_1,\ldots,Z_{A/N})$ satisfies
\[
\pr{Y\neq Z} = \|\LL(Y) - \LL(Z)\|_{\rm{TV}} \leq \frac{A}{N}e^{-c N}.
\]
Then we take $V_i=W_i$ if $Y_i = Z_i$, otherwise we take $V_i$ and $W_i$ to be independent. Hence this gives that 
\begin{align}\label{eq:name}
\|\LL((V_{iN})_{i=1}^{A/N}) - \LL((W_i)_{i=1}^{A/N})\|_{\rm{TV}} \leq \|\LL(Y) - \LL(Z) \|_{\rm{TV}} \leq \frac{A}{N} e^{-c N}.
\end{align}
By decreasing the value of $c > 0$, the above is $\lesssim e^{-c N}$.
Note that for any two measures $\mu_1$ and $\mu_2$ we have for any event $D$ that
\[
\mu_1(D) \leq \mu_2(D) + \|\mu_1-\mu_2\|_{\rm{TV}}.
\]
Thus letting $K=\frac{t'}{N}$, by~\eqref{eq:name} we have
\begin{align}\label{eq:vin}
\pr{\sum_{i=1}^{A/N}V_{iN}>K} \leq \pr{\sum_{i=1}^{A/N}W_i>K}  + e^{-c N}.
\end{align}
Since $Z_i\sim\til{\pi}$, it follows that $\E{W_i}=T_{r,R}$ for all $i$. Using 
Kac's moment formula \cite{FP_kac} we obtain for all $j \in \N$ and a positive constant $c$ 
\[
\E{W_1^j} \leq j! c^j T_{r,R}^j.
\]
Thus for $\theta>0$ we have
\[
\E{e^{\theta W_1}} \leq 1 +  \theta T_{r,R} +  \sum_{j=2}^{\infty} (c\theta T_{r,R})^j.
\]
Choosing $\theta = c_1 \delta /T_{r,R}$ we get that
\[
\E{e^{\theta W_1}} \leq 1 +  c_1\delta + \frac{(cc_1\delta)^2}{1- c c_1 \delta} \leq \exp\left(c_1\delta +\frac{(cc_1\delta)^2}{1-c c_1 \delta}  \right),
\]
and hence
\begin{align*}
\pr{\sum_{i=1}^{A/N} W_i >K} &\leq \left(\E{e^{\theta W_1}}\right)^{A/N} \exp\left(-\theta K \right) \leq \exp\left(\frac{A}{N}\left(c_1\delta +\frac{(cc_1\delta)^2}{1- c c_1 \delta}  \right) -\frac{c_1 \delta K}{T_{r,R}} \right).
\end{align*}
Since $\delta  r^{d-2}n^{-\psi-1/2}\leq 1$ and $\delta n^\psi
\leq 1$, substituting the values of $A$ and $K$ and choosing $c_1>0$ sufficiently small we get that for $n$ sufficiently large 
\[
\pr{\sum_{i=1}^{A/N} W_i >K} \lesssim e^{-c'\delta^2 A/N},
\]
where $c'$ is a positive constant. Hence this together with \eqref{eq:firstterm}, \eqref{eq:vit'}, and~\eqref{eq:vin} proves~\eqref{eq:lowergoal}.

Next we show that
\begin{align}\label{eq:n1a'}
\pr{N_1>A'} \lesssim Ne^{-c'\delta^2  A/N} +  e^{-cN}.
\end{align}
By the definition of $N_1$ again we get
\[
\pr{N_1>A'} = \pr{\sum_{i=1}^{A'}V_i+ (\sigma_0 -\tau_0)<t} \leq \pr{\sum_{i=N}^{A'}V_i<t}.
\]
Using the same coupling as before, it suffices to prove that there exists a positive constant $c'$ such that
\begin{align*}
\pr{\sum_{i=1}^{A'/N}W_i <\frac{(1-\delta)T_{r,R}A'}{N}} \lesssim e^{-c'\delta^2t/(T_{r,R}N)},
\end{align*}
where $(W_i)_i$ are i.i.d.\ excursion lengths started from i.i.d.\ points $(Z_i)_i$ distributed according to $\til{\pi}$. 
By Chernoff's bound we have for $\theta>0$ that
\begin{align}\label{eq:cher}
\pr{\sum_{i=1}^{A'/N}W_i <\frac{(1-\delta)T_{r,R}A'}{N}} \leq \left(\E{e^{-\theta W_1}}\right)^{A'/N} e^{\theta (1-\delta)T_{r,R}A'/N}.
\end{align}
Using that $e^{-x} \leq 1-x+x^2$ and that $\E{W_1^2} \leq c T_{r,R}^2$ by Kac's moment formula \cite{FP_kac}, we have
\begin{align*}
\E{e^{-\theta W_1}} \leq 1-\theta T_{r,R} +\theta^2\E{W_1^2} \leq 1-\theta T_{r,R} + c\theta^2T_{r,R}^2 \leq \exp\left(-\theta T_{r,R} + c\theta^2T_{r,R}^2\right).
\end{align*}
By taking $\theta = c_1\delta/T_{r,R}$ and plugging everything into~\eqref{eq:cher} we deduce
\[
\pr{\sum_{i=1}^{A'/N}W_i <\frac{t}{N}} \leq \exp\left(-\frac{A'}{N}\delta^2 c_1(1-cc_1)\right).
\]
Choosing $c_1>0$ small enough makes $1-cc_1$ positive, hence
\[
\pr{\sum_{i=1}^{A'/N}W_i <\frac{t}{N}} \lesssim e^{-c\delta^2 A'/N}.
\]
Recalling that $A$ and $A'$ are up to constants equal to $r^{d-2}\log n$ by Lemma~\ref{lem:trRlem}, the result follows by combining~\eqref{eq:lowergoal} and~\eqref{eq:n1a'}. 
\end{proof}

\begin{definition}\label{def:w}\rm{
Fix $\beta\in(0,1)$. 
We let $W$ be a random variable whose law is equal to that of the number of excursions the random walk makes across the annulus $\SS(0,n^\beta + n^\phi) \setminus \SS(0,n^{\beta})$ during one excursion across $\ball{10 n^\beta} \setminus \SS(0,n^\beta)$ when the starting point of the excursion on $\partial \ball{10n^\beta}$ is chosen according to $\til{\pi}$ from Lemma~\ref{lem:coupling}.
}
\end{definition}

In the proofs of Theorem~\ref{thm:uniform} and~\ref{thm:exact} we will take $\beta=\alpha-\epsilon$ for some small $\epsilon>0$. We suppress the dependency of $W$ on $\beta$ to lighten the notation.

\begin{lemma}\label{cl:hit}
The random variable~$W$ defined above is stochastically dominated by the sum of~$2d$ independent geometric random variables of parameter $n^{\phi-\beta}$ and satisfies
\[
\E{W} \asymp n^{\beta-\phi}.
\]
\end{lemma}

\begin{proof}
We start by proving that $\E{W} \gtrsim n^{\beta-\phi}$. We note that $\til{\pi}$ is up to multiplicative constants the same as the uniform distribution on $\partial\ball{10n^\beta}$ \cite[Lemma~6.3.7]{LawLim}. We can realize the random walk $X$ in the following way: let $U$ be a simple random walk on $\Z$ and $V$ be a simple random walk on $\Z^{d-1}$ which is independent of~$U$. Let $\xi(i)$ be i.i.d.\ Bernoulli random variables with success probability $(d-1)/d$. Write $r(k) =\sum_{i=1}^{k}\xi(i)$ and set 
\[
Z(k) = (U(k-r(k)),V(r(k))).
\]
Then it is elementary to check that $Z$ is a simple random walk in $\Z^d$, and hence $X(k)=Z(k)\bmod n$ is a simple random walk on $\Z_n^d$.

Let $x_0$ be the center of the side of the box which intersects the positive part of the first coordinate axis and 
let $A$ be the set of points of $\partial \SS(0,n^\beta+n^\phi)$ that are within distance $n^\beta/16$ of $x_0$. Then if $\tau$ is the first hitting time of $\partial\SS(0,n^\beta+n^\phi)$ after having first hit $\partial\SS(0,n^\beta)$, then it is easy to see that
\[
\pr{X(\tau) \in A} \geq p_0,
\]
where $p_0$ is a positive constant. Indeed, it is a standard fact that with positive probability Brownian motion stays close to a given continuous function $f:[0,1]\to \R^d$ for all times $t\in [0, 1]$. Hence the above claim is true for a Brownian motion started uniformly on $\partial\ball{10n^\beta}$. The result for random walk follows by Donker's invariance principle \cite[Theorem~8.6.5]{Durrett_prob}.

We now let 
\[
T=\min\left\{t\geq \tau: |V(r(t))-V(r(\tau))|\geq \frac{n^\beta}{4}\right\},
\]
i.e.\ $T$ is the first time that $V(r(\cdot))$ reaches distance $n^\beta/4$ from where it hit $\partial\SS(0,n^\beta+n^\phi)$ at time $r(\tau)$.  Let $s(t) = t - r(t)$.  Note that $s(T) - s(\tau)$ gives the number of steps that the random walk makes in the first coordinate axis during the time interval from $\tau$ to $T$.  Then there exist positive constants $p_1$ and $c_d$ depending only on $d$ such that
\begin{align}\label{eq:timeT}
\pr{s(T)-s(\tau)\geq c_dn^{2\beta}} \geq p_1.
\end{align}
On the event $\{X(\tau) \in A\}$ the random variable $W$ is greater than or equal to the number $E$ of excursions that $U$ makes from $n^\beta$ to $n^\beta+n^\phi$ before time $T$. Then using~\eqref{eq:timeT} we get that for all $u$
\begin{align*}
\pr{E\geq u}
\geq& \pr{E\geq u, s(T)-s(\tau)\geq c_dn^{2\beta}, X(\tau) \in A}\\
\gtrsim& \prcond{E\geq u}{ s(T)-s(\tau)\geq c_d n^{2\beta}}{}.
\end{align*}
Since $U$ is independent of $V$, on the event $s(T)-s(\tau)\geq c_d n^{2\beta}$, the random variable $E$ stochastically dominates the number of excursions that a one dimensional walk started from $0$ makes from $0$ to $n^\phi$ until time $c_dn^{2\beta}$. It  now immediately follows that 
\[
\E{E}\gtrsim n^{\beta-\phi}.
\]
We now turn to show the first assertion of the lemma.
Let $(Z^1,\ldots,Z^d)$ be a simple random walk in $\Z^d$. For $i=1,\ldots,d$, we let
\begin{itemize}
\item $A_i$ be the number of excursions that $Z^i$ makes from $-\frac{n^\beta}{2}$ to $-\frac{n^\beta}{2} - \frac{n^\phi}{2}$ before hitting $\pm 10n^{\beta}$ 
\item $B_i$ be the number of excursions that $Z^i$ makes from $\frac{n^\beta}{2}$ to $\frac{n^\beta}{2}+\frac{n^\phi}{2}$ before hitting~$\pm 10n^\beta$.
\end{itemize}
It is not hard to see that once the random walk hits $\partial\SS(0,n^\beta+n^\phi)$, then the number of excursions it makes from $\partial\SS(0,n^\beta)$ to $\partial\SS(0,n^\beta+n^\phi)$  before hitting $\partial\ball{10n^\beta}$ is stochastically dominated by 
\[
\sum_{i=1}^{d}(A_i + B_i).
\]
It follows from the gambler's ruin estimate that the $A_i$'s and $B_i$'s are geometric of parameter $n^{\phi-\beta}$, hence this completes the proof of the lemma.
\end{proof}

\begin{claim}\label{cl:geom}
Let $X$ be a geometric random variable of success probability $p\in (0,1/2]$ taking values in $\{1,2,\ldots\}$. Then for all $j$ we have
\[
\E{X^j} \lesssim \frac{j!}{p^j}.
\]
\end{claim}
\begin{proof}
See Appendix~\ref{app:elementary}.
\end{proof}

\begin{lemma}\label{lem:n2}
For each $\psi\in (0,1/2)$ there exists $n_0\geq 1$ and a positive constant~$c$ such that for all $n\geq n_0$ the following is true.
Fix $\beta, \phi\in(0,1)$ and~$t\asymp n^d\log n$. For all $\delta>0$ such that $\delta n^{\beta(d-2)-\psi-1/2}\leq 1$ and $\delta n^{\psi}\leq 1$ we let
\begin{align}\label{eq:defe}
\underline{E}(t,\delta) = \frac{t\E{W}}{(1+\delta)\trs_{n^\beta,10n^\beta}}\quad \text{and} \quad \overline{E}(t,\delta) = \frac{t\E{W}}{(1-\delta)\trs_{n^\beta,10n^\beta}}.
\end{align}
Then for all $x$ we have
\begin{align*}
\pr{\nsqsq_x(n^\beta,n^\beta+n^\phi,t)\notin [\underline{E}(t,\delta),\overline{E}(t,\delta)]}
\lesssim  n^\psi \exp\left(-c\delta^2 n^{\beta(d-2)-\psi}\right)  + e^{-cn^\psi}.
\end{align*}
\end{lemma}

\begin{proof}

To simplify notation throughout the proof we write $B=\underline{E}(t,\delta)$, $B'=\overline{E}(t,\delta)$, $N_1 =  \nsq_x(n^\beta,10n^\beta,t)$, and $N_2=\nsqsq_x(n^\beta,n^\beta+n^\phi,t)$. Let $N=k_0n^\psi$, $A$, and $A'$ be as in Lemma~\ref{lem:n1} with $r=n^\beta$, $R=10n^\beta$ and $\delta$ replaced by $\delta/2$.
We start with the upper bound. We have
\begin{align*}
\pr{N_2<B} \leq \pr{N_1<A} + \pr{N_2<B,N_1>A}.
\end{align*}
The first probability can be bounded using Lemma~\ref{lem:n1}. We first notice that all excursions across $\SS(x,n^\beta+n^\phi) \setminus\SS(x,n^\beta)$ are contained in the excursions across $\B(x,10n^\beta) \setminus \SS(x,n^\beta)$. Hence it follows that we can bound the second probability by the probability that in the first~$A$ excursions of the annulus $\B(x,10n^\beta)\setminus\SS(x,n^\beta)$ the number of excursions from $\partial\SS(x,n^\beta)$ to $\partial\SS(x,n^\beta+n^\phi)$ is at most~$B$. Let $W_i$ be the number of excursions across the ``thin'' annulus (i.e.\ $\SS(x,n^\beta+n^\phi)\setminus\SS(x,n^\beta)$) during the $i$-th excursion across the ``big'' annulus (i.e.\ $\B(x,10n^\beta)\setminus\SS(x,n^\beta)$). We first show
\begin{align}\label{eq:neweq}
\pr{\sum_{i=1}^{A}W_i<B} \lesssim Ne^{-c\delta^2 n^{\beta(d-2) -\psi}} + e^{-cN}.
\end{align}
By a union bound and the strong Markov property we get
\begin{align}\label{eq:unionbound}
\pr{\sum_{i=1}^{A}W_i<B} \leq \pr{\sum_{i=N}^{A}W_i<B} \leq N \max_x \prstart{\sum_{i=1}^{A/N}W_{iN}<\frac{B}{N}}{x}
\end{align}
Let $(Z_i)_i$ be i.i.d.\ distributed according to $\til{\pi}$ on $\partial\ball{10n^\beta}$ and let $(V_i)_i$ be i.i.d.\ with the same distribution as $W$ when the starting point of the excursion on $\partial\ball{10n^{\beta}}$ is $Z_i$. Let $(Y_i)_i$ be the exit points of the excursions of the random walk. 
Then under the optimal coupling of $Y=(Y_{N},\ldots,Y_{A/N})$ and $Z=(Z_1,\ldots,Z_{A/N})$ we get from Lemma~\ref{lem:coupling}
\[
\pr{Y\neq Z} = \|\LL(Y) - \LL(Z)\|_{\rm{TV}} \leq \frac{A}{N}e^{-c N}.
\]
Thus we can couple $(W_i)_i$ with $(V_i)_i$ by letting $V_i=W_i$ if $Y_i=Z_i$ and otherwise taking $V_i$ and $W_i$ to be independent.  This now gives
\[
\|\LL((W_i)_{i=1}^{A/N}) - \LL((V_i)_{i=1}^{A/N})\|_{\rm{TV}} \leq \|\LL(Y) - \LL(Z)\|_{\rm{TV}} \leq   \frac{A}{N} e^{-c N}.
\]
We obtain
\begin{align}\label{eq:obtain}
\pr{\sum_{i=1}^{A/N}W_{iN}<\frac{B}{N}}\leq \pr{\sum_{i=1}^{A/N}V_i<\frac{B}{N}} + \frac{A}{N} e^{-c N}.
\end{align}
By adjusting the value of $c > 0$, the error term above is $\lesssim e^{-c N}$.  So now we need to bound the probability appearing on the right hand side of \eqref{eq:obtain}.
Applying Chernoff's inequality we get for $\theta>0$
\begin{align}
\label{eq:cherb}
\pr{\sum_{i=1}^{A/N}V_i<\frac{B}{N}} \leq \E{e^{-\theta\sum_{i=1}^{A/N}V_i}} e^{\theta B/N} =\E{e^{-\theta W}}^{A/N} e^{\theta B/N},
\end{align}
where the last step follows since the $(V_i)_i$ are i.i.d.\ with $V_i\sim W$ for all $i$. Using the inequalities 
\[
e^{-x} \leq 1-x + \frac{x^2}{2}\quad \text{for} \quad x\geq 0 \quad \text{and} \quad e^x\geq 1+x \quad \text{for all} \quad x \in \R. 
\]
we obtain 
\begin{align}
\label{eq:lap}
\E{e^{-\theta W}} \leq \E{1-\theta W + \frac{\theta^2}{2} W^2} \leq \exp\left(-\theta \E{W}+ \frac{\theta^2}{2}\E{W^2} \right).
\end{align}
Combining \eqref{eq:cherb} and \eqref{eq:lap} we thus have that
\begin{align*}
\pr{\sum_{i=1}^{A/N}V_i<\frac{B}{N}} &\leq \exp\left(\frac{A}{N}\left(-\theta\E{W}+\frac{\theta^2}{2}\E{W^2} \right)+\theta \frac{B}{N} \right) \\
&= \exp\left( \frac{A\theta^2\E{W^2}}{2N}- \frac{A\theta  \E{W} \eta}{N}  \right),
\end{align*}
where $\eta=\delta/(2(1+\delta))$.
Setting $\theta = \frac{\eta \E{W}}{\E{W^2}}$, we deduce that
\[
\pr{\sum_{i=1}^{A/N}V_i<\frac{B}{N}} \leq \exp\left( -\frac{A\eta^2 \E{W}^2}{2N\E{W^2}} \right).
\]
From Lemma~\ref{cl:hit} and~Claim~\ref{cl:geom} we see that there exists a positive constant $c$ such that $\E{W}^2/\E{W^2} \geq c$.  This implies that there exists a positive constant $c'$ such that
\[
\pr{\sum_{i=1}^{A/N}V_i<\frac{B}{N}} \leq \exp\left( -\frac{c'A\delta^2}{N} \right).
\]
Since $A\asymp n^{\beta(d-2)}\log n$ by Lemma~\ref{lem:trRlem}, the above together with~\eqref{eq:unionbound} and~\eqref{eq:obtain} proves~\eqref{eq:neweq} and this completes the proof of the upper bound.

For the lower bound 
in the same way as above we have
\[
\pr{N_2>B'} \leq \pr{N_1>A'}   + \pr{N_2>B', N_1<A'}.
\]
For the first term we use Lemma~\ref{lem:n1}. For the second term we replace again this event by the event that in the first $A'$ excursions across the ``big'' annulus there were at least $B'$ excursions across the ``thin''  one.
Hence if $(W_i)_i$ are as before, setting $H=N^2\E{W}$ we have
\begin{align*}
\pr{N_2>B',N_1<A'}
&\leq \pr{\sum_{i=1}^{A'}W_i>B'}\\
& \leq N \max_x \prstart{\sum_{i=1}^{A'/N}W_{iN}>\frac{B'}{N} -\frac{H}{N} }{x} + \pr{\sum_{i=1}^{N-1}W_i>H}.
\end{align*}
From Lemma~\ref{cl:hit} we immediately get that 
\begin{align*}
\pr{\sum_{i=1}^{N-1}W_i>H} \leq \pr{\sum_{i=1}^{N-1}G_i>H},
\end{align*}
where $(G_i)_i$ are i.i.d.\ each having the law of the sum of $2d$ independent geometric random variables of success probability $n^{\phi-\beta}$. Using Claim~\ref{cl:geom} we then get that for a positive constant $c$ that
\[
\pr{\sum_{i=1}^{N-1}W_i>H} \lesssim e^{-cN}.
\]
Using the same coupling as before we obtain
\begin{align*}
\pr{\sum_{i=1}^{A'/N}W_{iN}>\frac{B'-H}{N}} \leq \pr{\sum_{i=1}^{A'/N}V_{i}>\frac{B'-H}{N}}  + \frac{A'}{N}e^{-cN},
\end{align*}
where the $(V_i)_i$ are i.i.d.\ and distributed according to the law of $W$.  By possibly decreasing the value of $c > 0$, the error term above is $\lesssim e^{-c N}$.
By Lemma~\ref{cl:hit} and Claim~\ref{cl:geom} we have for a positive constant $c_1$ that
\[
\E{e^{\theta W}} = 1+\theta \E{W} + \sum_{j=2}^{\infty}\frac{\theta^j \E{W^j}}{j!} \leq \exp\left( \theta\E{W} + \frac{(c_1 \theta \E{W})^2}{1-c_1\theta \E{W}} \right).
\]
Let $\eta=\delta/(2(1-\delta))$.  Using the above, Chernoff's inequality, and substituting the expression for $B'$ gives
\begin{align*}
&\pr{\sum_{i=1}^{A'/N}V_i >\frac{B'-H}{N}}\leq \E{e^{\theta W}}^{A'/N} e^{-\theta (B'-H)/N} \\
\leq& \exp\left(\frac{\theta H}{N} \right)\exp\left(\frac{A'}{N}\left( \theta\E{W} +\frac{(c_1 \theta \E{W})^2}{1-c_1\theta \E{W}} \right) \right) \exp\left( -\frac{A'\theta(1+\eta)\E{W}}{N} \right).
\end{align*}
Setting $\theta=c_2\eta/\E{W}$ for a positive constant $c_2$ to be determined and recalling that $H=N^2\E{W}$ we get
\begin{align*}
\pr{\sum_{i=1}^{A'/N}V_i >\frac{B'-H}{N}}&\leq \exp(c_2\eta N)\exp\left(-\frac{ c_2\eta^2A'}{N} \left(1-\frac{c_1^2 c_2}{1-c_1c_2\eta} \right)  \right).
\end{align*}
Using the assumption $\delta n^\psi\leq 1$ and taking $c_2 >0$ sufficiently small we get for a positive constant~$c'$ and all sufficiently large $n$ that
\begin{align*}
\pr{\sum_{i=1}^{A'/N}V_i >\frac{B'-H}{N}}\lesssim e^{-c'\delta^2 A'/N}
\end{align*}
and, since $A'\asymp n^{\beta(d-2)}\log n$ by Lemma~\ref{lem:trRlem}, this finishes the proof of the lemma.
\end{proof}

\begin{definition}\label{def:boxes}
\rm{
Fix $\phi,\beta \in (0,1)$.  Let $\ol{\SS}_\beta$ be a partition of $\Z_n^d$ into (disjoint) boxes of side length $n^\beta + n^\phi$ (we will suppress the dependency on $\phi$).  For each $\ol{B} \in \ol{\SS}_\beta$ we let $B$ (resp.\ $\ul{B}$) be the box of side length $n^\beta$ (resp.\ $n^\beta - n^\phi$) which is concentric with $\ol{B}$ and we let $\SS_\beta$ (resp.\ $\ul{\SS}_\beta$) be the collection of all such concentric boxes with this side length.  For each $z \in \cup_{B \in \SS_\beta} B$ we let $S_z$ be the element of $\SS_\beta$ which contains $z$ and $\ol{S}_z$ the element of~$\ol{\SS}_\beta$ which contains $z$.
We let $\AA = \Z_n^d\setminus\cup_{S \in \SS_\beta} \ul{S}$ be the collection of points of the torus that lie in the annuli between the boxes of side length~$n^\beta+n^\phi$ and the concentric boxes of side length~$n^\beta-n^\phi$.
}
\end{definition}

\begin{definition}\label{def:n3}
\rm{
Fix $\phi,\beta \in (0,1)$ and recall the definition of $\underline{E}$ from Lemma~\ref{lem:n2}. 
For every~$z\in \Z_n^d\setminus \AA$ and~$R>r$ we define $N_z(r,R,t)$ to be the number of excursions across the annulus $\B(z,R)\setminus\B(z,r)$ during the first $\underline{E}(t,\delta/4)$ excursions across the annulus $\ol{S}_z \setminus S_z$ where $S_z$ and $\ol{S}_z$ are as in Definition~\ref{def:boxes}.
}
\end{definition}

\begin{lemma}\label{lem:n3}
For each $\psi\in(0,1/2)$ and $\beta \in (0,1)$ there exist $n_0 \geq 1$ and a positive constant $c$ such that for all $n \geq n_0$ the following is true.  Let $n^\beta \geq R\geq 10r$ and $\delta\in(0,1/3)$ satisfy 
\[
\delta n^\psi \leq 1 \quad \text{and} \quad  \delta n^{\beta (d-2)-\psi-1/2} \leq 1.
\]
If $t \asymp n^d\log n$, then for all $z\in \Z_n^d\setminus \AA$ we have that 
\[
\pr{N_z(r,R,t) \notin [\underline{L}(t),\overline{L}(t)]} \lesssim n^\psi \exp\left(-c \delta^2 r^{d-2}n^{-\psi}\right) +    e^{-c n^\psi},
\]
where $\underline{L}(t)= \frac{t}{(1+\delta)\trb_{r,R}}$ and $\overline{L}(t)=\frac{t}{(1-\delta)\trb_{r,R}}$.
\end{lemma}

\begin{proof}
We define $\til{N}_z$ to be the number of excursions across the annulus $\B(z,R) \setminus \B(z,r)$ up to time $(1-\delta/2)t$ and we let $T$ be the time it took for the $\underline{E}(t,\delta/4)$ excursions across the ``thin'' annulus $\SS(z,n^\beta+n^\phi)\setminus \SS(z,n^\beta)$ to complete. 
Notice that on the event $\{T\geq (1-\delta/2)t\}$ we have $\til{N}_z\leq N_z$ hence we get
\begin{align}\label{eq:nzlt}
\pr{N_z< \underline{L}(t)} \leq \pr{T<(1-\delta/2)t} + \pr{\til{N}_z< \underline{L}(t)}.
\end{align}
We recall the definition of $\underline{E}(t,\delta/4)$
\[
\underline{E}(t,\delta/4)=\frac{\E{W}t}{(1+\delta/4)\trs_{n^\beta,10n^\beta}}.
\]
The first probability on the right side of \eqref{eq:nzlt} can be written as 
\begin{align}\label{eq:txi}
\pr{T< (1-\delta/2)t} =  \pr{N_2>\underline{E}(t,\delta/4)}.
\end{align}
Let
\begin{align*}
\Gamma
=\frac{(1-\delta/2)\E{W}t}{(1-\delta/8)\trs_{n^\beta,10n^\beta}}
  =\frac{(1-\delta/2)(1+\delta/4)}{1-
\delta/8}  \underline{E}(t,\delta/4) <\underline{E}(t,\delta/4).
\end{align*}
Let $N_2 = \nsq_z(n^\beta,n^\beta+n^\phi,(1-\delta/2)t)$.
Applying Lemma~\ref{lem:n2} we get that
\begin{align}\label{eq:gamma}
\pr{N_2>\underline{E}(t,\delta/4)} \leq \pr{N_2>\Gamma} \lesssim  n^\psi \exp\left(-c \delta^2 n^{\beta(d-2)-\psi} \right) +  e^{-c n^\psi}.
\end{align}
For the second probability on the right side of \eqref{eq:nzlt}, we apply Lemma~\ref{lem:n1} to obtain for all $\delta\in(0,1/3)$ that
\begin{align}\label{eq:n3lt}
\pr{\til{N}_z < \underline{L}(t)} \leq  \pr{\til{N}_z<\frac{(1-\delta/2)t}{(1+\delta/3)\trb_{r,R}}}\lesssim
e^{-c \delta^2 r^{d-2}/n^\psi}+e^{-c n^\psi}.
\end{align}
Combining~\eqref{eq:nzlt}, \eqref{eq:txi}, and~\eqref{eq:n3lt} we deduce
\begin{align}\label{eq:nzfinal}
\pr{N_z< \underline{L}(t)} \lesssim n^\psi e^{-c \delta^2 r^{d-2}/n^\psi} +    e^{-c n^\psi}
\end{align}
and this finishes the proof of the first part.

We define $N_z'$ to be the number of excursions across $\B(z,R) \setminus \B(z,r)$ by time~$(1+\delta/2)t$. Let~$T$ be as in the first part of the proof.
Notice that on the event $\{T<(1+\delta/2)t\}$ we have $N_z'\geq N_z$. So
\begin{align}
\pr{N_z\geq \overline{L}(t)} \leq \pr{T\geq (1+\delta/2)t} +\pr{N_z'\geq \overline{L}(t)}. \label{eq:probnz}
\end{align}
By the definition of $T$ we have
\begin{align}
 \pr{T\geq (1+\delta/2)t}= \pr{N_2'\leq \underline{E}(t,\delta/4)}. \label{eq:probnz2}
\end{align}
Applying Lemma~\ref{lem:n2} we get that if 
\[
\Gamma = \frac{(1+\delta/2) t \E{W}}{(1+\delta/4)\trs_{n^\beta,10n^\beta}},
\]
then writing $N_2' = \nsq_z(n^\beta,n^\beta+n^\phi, (1+\delta/2)t)$ we have
\begin{align}
\pr{N_2'\leq \Gamma} \lesssim  n^\psi e^{-c \delta^2 n^{\beta(d-2)-\psi}} + e^{-c n^\psi}. \label{eq:probnz3}
\end{align}
It is now easy to see that for all $\delta >0$ we have $\Gamma>\underline{E}(t,\delta/4)$, and hence combining \eqref{eq:probnz2} and \eqref{eq:probnz3} we obtain the following bound for the first probability on the right side of \eqref{eq:probnz}:
\begin{align}
\pr{N_2' \leq \underline{E}(t)} \leq \pr{N_2' \leq \Gamma}
\lesssim  n^\psi e^{-c \delta^2 n^{\beta(d-2)-\psi}} + e^{-c n^\psi}. \label{eq:probnz4}
\end{align}
By Lemma~\ref{lem:n1} we can bound the second probability on the right side of \eqref{eq:probnz} by:
\begin{equation}
\label{eq:probnz5}
\pr{N_z'\geq \overline{L}(t)} \leq \pr{N_z'\geq \frac{(1+\delta/2)t}{(1-\delta/4)\trb_{r,R}}} \lesssim n^\psi e^{-c \delta^2 r^{d-2}/n^\psi} +  e^{-c n^\psi}.
\end{equation}
Inserting the bounds from \eqref{eq:probnz4} and \eqref{eq:probnz5} into \eqref{eq:probnz} concludes the proof.
\end{proof}

\section{Hitting probabilities}\label{sec:hitting}

In this section we collect some results about hitting probabilities of simple random walks in $\Z_n^d$ for $d\geq 3$. Some of the proofs are deferred to Appendix~\ref{app:elementary}.  We start by recalling Harnack's inequality (see e.g.~\cite[Theorem~6.3.8]{LawLim}). 

\begin{lemma}[Harnack's inequality]\label{lem:harnackineq}
Fix $R\geq 2r$ and let $f$ be a positive harmonic function on $\B(0,R) \subseteq \Z^d$. Then for all $x,y\in \B(0,r)$ we have
\begin{align*}
\frac{f(x)}{f(y)} = 1 + O\left(\frac{r}{R}\right).
\end{align*}
\end{lemma}

\begin{proof}
See Appendix~\ref{app:elementary}.
\end{proof}

\begin{lemma}\label{lem:hittingprob}
There exists a constant $C_d > 0$ depending only on $d$ such that the following is true.  Let $n/4 \geq R\geq 2r$ such that both $r,R$ tend to infinity as $n\to \infty$ and let $z\in \Z_n^d$ with $\|z\| \leq r/4$. We denote by~$\tau_R$ the first hitting time of $\partial\B(0,R)$ and by $\tau_z$ the first hitting time of $z$. Then for all $x\in \partial\B(0,r)$ and all $y\in \partial\B(0,R)$ we have
\begin{align*}
\prcond{\tau_z<\tau_R}{X(\tau_R)=y}{x} = \frac{\crw}{r^{d-2}}\left(1+O\left(\frac{r}{R} \right) +O\left(\frac{1}{r^2} \right) +O\left( \frac{\norm{z}}{r} \right) \right).
\end{align*}
\end{lemma}

\begin{proof}
See Appendix~\ref{app:elementary}.
\end{proof}

\begin{remark}\rm{The constant $\crw$ from the statement of Lemma~\ref{lem:hittingprob} is given by $c_d/G(0)$, where $c_d$ is the constant from~\cite[Theorem~4.3.1]{LawLim} and $G$ is the Green's function for simple random walk on $\Z^d$. That is, $G(0)$ is equal to the expected number of visits to $0$ made by simple random walk started from $0$ before escaping to~$\infty$.
}
\end{remark}

\begin{definition}\label{def:pddef}\rm{
We define $p_d$ to be the probability that a simple random walk on~$\Z^d$ started from $0$ returns to~$0$. 
}
\end{definition}

\begin{remark}\label{rem:pd}\rm{
For $d=3$, it is well-known (see e.g.~\cite{Spitzer}) that $p_3\approx 0.34$. It is also easy to see that $p_d\to 0$ as $d\to \infty$. Note that $p_d$ is equal to the probability that a simple random walk in~$\Z^d$ starting from~$0$ visits a given neighbour of~$0$ before escaping to~$\infty$. 
}
\end{remark}

\begin{lemma}\label{lem:2pointestimate}
Let $n/4 \geq R>2r \to \infty$ and $x,y\in \Z_n^d$ satisfying $\norm{x-y}=o(r)$. 
We denote by $\tau_R$ the first hitting time of~$\B(x,R)$ and by $\tau_x$ (resp.\ $\tau_y$) the first hitting time of~$x$ (resp.~$y$).
Then for all $a\in \partial\B(x,r)$ and all $b\in \partial\B(x,R)$  then we have
\begin{align}
\label{eq:xyafterR}
\prcond{\tau_x\wedge\tau_y<\tau_R}{X(\tau_R)=b}{a} &\geq \frac{2\crw}{(1+p_d)r^{d-2}}\left( 1 +o(1)+O\left(\frac{r}{R} \right) + O\left(\frac{1}{r^2} \right) \right).
\intertext{Moreover, if $x$ and $y$ are neighbours, then we have}
\label{eq:xyafterRneighbors}
\prcond{\tau_x\wedge\tau_y<\tau_R}{X(\tau_R)=b}{a} &=  \frac{2\crw}{(1+p_d)r^{d-2}}\left(1 + O\left( \frac{r}{R} \right) + O\left(\frac{1}{r^2} \right)\right).
\end{align}
\end{lemma}

\begin{proof}
By Bayes' formula we have
\begin{align*}
\prcond{\tau_x \wedge \tau_y < \tau_R}{X(\tau_R) = b}{a} &= \frac{\prcond{X_{\tau_R}=b}{\tau_x\wedge \tau_y<\tau_R}{a}}{\prstart{X_{\tau_R}=b}{a}} \prstart{\tau_x\wedge \tau_y<\tau_R}{a} \\
&= \left(1+O\left(\frac{r}{R}\right) \right) \prstart{\tau_x\wedge\tau_y<\tau_R}{a},
\end{align*}
where the second equality follows by Harnack's inequality (Lemma~\ref{lem:harnackineq}).
Let 
\[ Z = \sum_{t=0}^{\tau_R} \1(X(t) \in \{x,y\})\]
be the number of times that $X$ visits either $x$ or $y$ before hitting $\partial \B(x,R)$.  Then it is easy to see that
\begin{align*}
\prstart{\tau_x\wedge\tau_y<\tau_R}{a} = \frac{\estart{Z}{a}}{\escond{Z}{\tau_x\wedge\tau_y<\tau_R}{a}}.
\end{align*}
Note that we can write
\begin{align}\label{eq:repzwrite}
Z=\sum_{t=0}^{\infty} \1(X(t)\in \{x,y\}) - \sum_{t=\tau_R}^{\infty}\1(X(t)\in\{x,y\}).
\end{align}
Applying~\cite[Theorem~4.3.1]{LawLim} and the strong Markov property we thus have
\begin{align*}
\estart{Z}{a} = \frac{2\crw G(0)}{r^{d-2}}(1+o(1)) + O\left(\frac{1}{R^{d-2}} \right) + O\left( \frac{1}{r^d}\right),
\end{align*}
where the $o(1)$ term disappears when $x$ and $y$ are neighbours. 
For the denominator we have
\begin{align*}
\escond{Z}{\tau_x\wedge\tau_y<\tau_R}{a} =& \estart{\sum_{t=0}^{\tau_R}\1(X(t)\in\{x,y\})}{x} \prcond{\tau_x<\tau_y}{\tau_x\wedge\tau_y<\tau_R}{a} \\+& \estart{\sum_{t=0}^{\tau_R}\1(X(t)\in\{x,y\})}{y} \prcond{\tau_y<\tau_x}{\tau_x\wedge\tau_y<\tau_R}{a}.
\end{align*}

Consequently, using the representation for $Z$ from~\eqref{eq:repzwrite} it is easy to see by applying~\cite[Theorem~4.3.1]{LawLim} again and the last part of Remark~\ref{rem:pd} that
\begin{align*}
\escond{Z}{\tau_x\wedge\tau_y<\tau_R}{a} \leq G(0)(1 +p_d) +  O\left(\frac{1}{R^{d-2}}   \right) + O\left(\frac{1}{r^d}\right)
\end{align*}
with equality when $x$ and $y$ are neighbours. Putting everything together yields the result.
\end{proof}

\section{Separated points}\label{sec:separated}

In this section we define the time~$t_*$ referred to in the Introduction and we prove that with high probability at time~$\alpha t_*$ for $\alpha \in (0,1)$ large enough the points in the last visited set are at distance at least $n^\gamma$ for some $\gamma$ to be defined later. We prove these results in a certain setup which we describe below in order to make them compatible with the proofs of Theorems~\ref{thm:uniform} and~\ref{thm:exact}.

\textbf{Setup:} Let $\beta = \alpha - \epsilon$ for some $\epsilon>0$ small enough to be determined later.
As in Definition~\ref{def:boxes}, we divide the torus into boxes of side length $n^\beta+n^\phi$ with $\phi \in(0,\beta)$ and we will make use of the notation described there.  For every $S\in \SS_\beta$ we write~$\tau_{S}$ for the first time that the random walk has made~$\underline{E}(\alpha t_*,\delta/4)=\alpha t_*\E{W}/((1+\delta/4)\trs_{n^\beta,10n^\beta})$ excursions across the annulus surrounding~$S$, where $W$ is as in Definition~\ref{def:w} and~
\begin{align}\label{eq:defdeltafinal}
\delta =r^{(2-d)/2} n^\psi = n^{\phi(2-d)/\kappa + \psi} \quad \text{where} \quad \kappa = d\wedge 6
\end{align}
and $\psi>0$ is very small and will be fixed later. We will explain the choice of the value of~$\delta$ in Remark~\ref{rem:explanation} in  Section~\ref{sec:totalvar}.
 
We recall $\AA = \Z_n^d\setminus\cup_{S \in \SS_\beta} \ul{S}$ is the collection of points of the torus that lie in the annuli between the boxes of side length~$n^\beta+n^\phi$ and the concentric boxes of side length~$n^\beta-n^\phi$.
 
As in Definition~\ref{def:boxes}, for every $z\notin \AA$, we write $S_z \in \SS_\beta$ for the unique box in $\SS_\beta$ that contains $z$. 
We now consider the process~$Y=(Y_z)_{z}$ defined by $Y_z=\1(\tau_z>\tau_{\SS_z})$ for $z\notin \AA$ and $Y_z=0$ for $z\in \AA$.

For any $\zeta>0$ we define the collection of $\zeta$-separated subsets,
$\SS(\zeta)$ as follows
\begin{align}\label{eq:separdef}
\SS(\zeta)=\{ U\subseteq \Z_n^d: \ \norm{x-y} \geq n^{\zeta}, \ \forall x,y\in U\}.
\end{align}
We will now define the time $t_*$ that was introduced in the statement of Theorem~\ref{thm:uniform} (but not defined there). We set 
\begin{align}\label{eq:deftstar}
t_*= \frac{ \log(n^d) \trb_{n^{2\phi/\kappa},n^\phi}}{\prstart{\tau_z<\tau_{n^\phi}}{\pi}},
\end{align}
where $\phi$ is as above. The precise value of~$\phi$ and the radii in~\eqref{eq:deftstar} are selected to optimize several error terms in Claim~\ref{cl:annulusx} and equation~\eqref{eq:condphi} in Section~\ref{sec:totalvar} and it is explained in Remark~\ref{rem:explanation}.

Note that we write $\prstart{\tau_z<\tau_{n^\phi}}{\pi}$ for the probability that $z$ is hit in an excursion across the annulus $\B(z,n^\phi)\setminus\B(z,n^{2\phi/\kappa})$ when the random walk starts from the stationary distribution.  (Lemma~\ref{lem:hittingprob} gives an error bound which is independent of the starting point.)

The following lemma implies that $t_* = \tcov(1+o(1))$ and it is proved in Appendix~\ref{app:hitting}.

\begin{lemma}\label{lem:int}
For all $r,R\to \infty$ with $R=o(n)$ and $r=o(R)$ as~$n\to \infty$ we have
\[
\frac{\trb_{r,R}}{\pr{\tau_z<\tau_R}} = \tmax (1+o(1))\quad \text{as} \quad n\to \infty.
\]
\end{lemma}

\begin{lemma}\label{lem:hit}
For every $x \in \Z_n^d$ we have
\[
\pr{\tau_x>\alpha t_*} \lesssim n^{-\alpha d}
\]
\end{lemma}

\begin{proof}

Let $\phi$ be as in the definition of~$t_*$ and $r=n^{2\phi/\kappa}$, $R=n^\phi$ and $N_x$ be the number of excursions across the annulus $\B(x,R)\setminus\B(x,r)$ before time $\alpha t_*$. Let $A = \alpha t_*/((1+\delta)\trb_{n^{2\phi/\kappa},n^\phi})$ and $\psi\in(0,1/2)$ be as in Lemma~\ref{lem:n1} and $\delta$ as in~\eqref{eq:defdeltafinal}.
Writing ${\rm Exc}(x,i) = \{x\text{ not hit in the $i$-th excursion}\}$, we then have
\begin{align*}
\pr{\tau_x>\alpha t_*} &\leq \pr{\tau_x>\alpha t_*, N_x\geq A} + \pr{N_x<A}\\
 &\leq \pr{\bigcap_{i=2}^{A}{\rm Exc}(x,i)} + \pr{N_x<A}. 
\end{align*}
We took the lower index in the intersection to be $2$ rather than $1$, because the first excursion has a positive  chance of starting in~$\B(x,r)$, while the second does not.
Let $a_i = X(\tau_i)$ and $b_i=X(\sigma_i)$, where $\tau_i, \sigma_i$ are defined at the beginning of Section~\ref{sec:excursions}. Let $\F=\sigma(\{(a_i,b_i):i=1,\ldots,A\})$ be the $\sigma$-algebra generated by $a_i, b_i$. Notice that conditional on $\F$ the events ${\rm Exc}(x,i)$ are independent for $i=2,\ldots, A$. Writing $\tau_R$ for the first hitting time of $\partial\B(x,R)$ we therefore get
\begin{align*}
\pr{\bigcap_{i=2}^{A}{\rm Exc}(x,i)} = \E{\prod_{i=2}^{A}\prcond{{\rm Exc}(x,i)}{\F}{}} 
= \E{\prod_{i=2}^{A}\prcond{\tau_x<\tau_R}{X(\tau_R)=b_i}{a_i}}.
\end{align*}
From Lemma~\ref{lem:hittingprob} we immediately get for all $i\geq 2$ that
\begin{align*}
\prcond{\tau_x<\tau_R}{X(\tau_R)=b_i}{a_i}= 1 - \frac{\crw}{n^{2\phi(d-2)/\kappa}} \left(1 + O\left(\frac{1}{n^{\phi (\kappa-2)/\kappa}}\right)\right).
\end{align*}
Hence we deduce 
\begin{align*}
&\pr{\bigcap_{i=2}^{A}{\rm Exc}(x,i)} = \left( 1 - \frac{\crw}{n^{2\phi(d-2)/\kappa}} \left(1 +O\left(\frac{1}{n^{\phi (\kappa-2)/\kappa}}\right)\right)\right)^{A-1} \\
&\leq \exp\left(-(A-1)\frac{\crw}{n^{2\phi(d-2)/\kappa}} \left(1+O\left(\frac{1}{n^{\phi (\kappa-2)/\kappa}}\right)\right)   \right)\\
&\leq \exp\left(-\alpha d\log n \left( 1 + O(\delta) + O\left(\frac{1}{n^{\phi (\kappa-2)/\kappa}}\right)\right)   \right) \exp\left(O\left(n^{-2\phi(d-2)/\kappa
}\right)
\right)
\\&=
n^{-\alpha d} \left( 1+ O\left(\delta\log n + n^{-\phi(\kappa-2)/\kappa}\log n   \right)\right),
\end{align*}
where for the second inequality we used the expression for $A$ and $t_*$ and Lemma~\ref{lem:hittingprob}.
Recalling that $\delta=n^{-(d-2)\phi/\kappa + \psi}$ and $\psi\in(0,1/2)$ small enough we thus see that
\begin{align*}
\pr{\bigcap_{i=2}^{A}{\rm Exc}(x,i)} \lesssim n^{-\alpha d}.
\end{align*}
By Lemma~\ref{lem:n1} (since the choice of $\delta$ satisfies the assumptions) we get
\begin{align*}
\pr{N_x<A}  \lesssim n^{-\alpha d}
\end{align*} 
and this concludes the proof.
\end{proof}

\begin{lemma}
\label{lem:yhit}
Fix $0<\zeta<\phi$ and $c>0$. 
Let $U\in \SS(\zeta)$ with $|U| \leq  c$. Then we have
\[
\E{\prod_{u\in U} Y_u} \lesssim \frac{1}{n^{\alpha d |U| (1+o(1))}},
\]
where the constant in $\lesssim$ depends only on $c$. Moreover, for any $u\in \Z_n^d$ we have
\[
\pr{Y_u = 1} \lesssim \frac{1}{n^{\alpha d}}.
\]
\end{lemma}

Note that the final part of Lemma~\ref{lem:yhit} is not the same as Lemma~\ref{lem:hit}, because we consider the hitting probability after the random walk has made a certain number of excursions across $\ol{S}_x\setminus S_x$ rather than at time $\alpha t_*$.

\begin{proof}[Proof of Lemma~\ref{lem:yhit}]

Around every $u\in U$ we place two balls of radii $r=n^{2\zeta/\kappa}$ and $R=\tfrac{1}{2} n^\zeta$. We let~$N_u$ be the number of excursions across the annulus that is created by the two balls during the first~$\underline{E}(\alpha t_*,\delta/4)$ excursions across the ``thin'' annulus $\ol{S}_u\setminus S_u$, where $\underline{E}$ is as in Lemma~\ref{lem:n2} and we will set the value of $\delta$ later in the proof.
We then have
\begin{align}\label{eq:unionupbound}
\E{\prod_{u\in U}Y_u} \leq \sum_{u\in U} \pr{N_u<\underline{L}(\alpha t_*)} + \E{\prod_{u\in U} Y_u\1(N_u>\underline{L}(\alpha t_*))},
\end{align}
where $\underline{L}(t)$ is defined in the statement of Lemma~\ref{lem:n3}.
We let~$\F$ be the $\sigma$-algebra generated by $X(\tau_i(u))$ and $X(\sigma_i(u))$ for all $u\in U$, where $\tau_i(u)$ and $\sigma_i(u)$ are defined at the beginning of Section~\ref{sec:excursions} with respect to the annuli~$\B(u,R)\setminus\B(u,r)$. Writing ${\rm{Exc}}(u,i) = \{u\text{ not hit in the } i \text{-th excursion}\}$ we have
\begin{align*}
\E{\prod_{u\in U}Y_u\1(N_u>\underline{L}(\alpha t_*))} \leq \E{\prcond{\bigcap_{u\in U}\bigcap_{i=2}^{\underline{L}(\alpha t_*)}{\rm{Exc}}(u,i)}{\F}{}}.
\end{align*}
Given $\F$ the events $\cap_{i=2}^{\underline{L}(\alpha t_*)}{\rm Exc}(u,i)$
are independent over different $u\in U$, and hence
\begin{align}\label{eq:indep}
\E{\prod_{u\in U}Y_u\1(N_u>\underline{L}(\alpha t_*))} \leq \E{\prod_{u\in U}\prcond{\bigcap_{i=2}^{\underline{L}(\alpha t_*)}{\rm Exc}(u,i)}{\F}{}}.
\end{align}
By Lemma~\ref{lem:hittingprob} we have
\begin{align*}
 \prcond{\bigcap_{i=2}^{\underline{L}(\alpha t_*)}{\rm Exc}(u,i)}{\F}{}
= \left( 1- \frac{\crw}{n^{2(d-2)\zeta/\kappa}}\left( 1 + O\left(\frac{1}{n^{\zeta (\kappa-2)/\kappa}} \right)  \right) \right)^{\underline{L}(\alpha t_*)-1} \\
\leq \exp\left( -\frac{\crw (\underline{L}(\alpha t_*) -1)}{n^{2(d-2)\zeta/\kappa}}\left( 1 + O\left(\frac{1}{n^{\zeta(\kappa-2)/\kappa}} \right) \right)   \right).
\end{align*}
We now set~$\delta = n^{-(d-2)\zeta/\kappa+\psi}$ and $\psi\in(0,1/2)$ very small.
Using Lemma~\ref{lem:int} we get
\[
\underline{L}(\alpha t_*) = \frac{\alpha n^{2(d-2)\zeta/\kappa} \log(n^d)}{\crw(1+\delta)} (1+o(1)) = \frac{\alpha n^{2(d-2)\zeta/\kappa} \log(n^d)}{\crw} (1+o(1)).
\]
Substituting this expression for $\underline{L}(\alpha t_*)$ in the inequality above we deduce
\begin{align}\label{eq:upbound}
\begin{split}
\prcond{\bigcap_{i=2}^{\underline{L}(\alpha t_*)}{\rm Exc}(u,i)}{\F}{}
\leq \exp\left( -\alpha\log(n^d) (1+o(1)) \right) 
\lesssim \frac{1}{n^{\alpha d(1+o(1))}}.
\end{split}
\end{align}
Lemma~\ref{lem:n3} together with~\eqref{eq:unionupbound}, \eqref{eq:indep} and~\eqref{eq:upbound} give
\begin{align*}
\E{\prod_{u\in U} Y_u} \lesssim \frac{1}{n^{\alpha d|U|(1+o(1))}}. 
\end{align*}
Note that in the above argument if $U=\{u\}$, then we can place two balls of radii~$n^{2\phi/\kappa}$ and~$n^\phi$ around $u$ and hence we lose the $1+o(1)$ term in the expression for $\underline{L}$. Therefore we get
\[
\pr{Y_u=1} \lesssim \frac{1}{n^{\alpha d}}
\]
and this concludes the proof.
\end{proof}

\begin{lemma}\label{lem:2points}

Fix $0<\zeta<\phi$ and $c>0$. Let $U\notin \SS(\zeta)$ with $|U| \leq c$. Suppose that $U$ viewed as a subset of the graph which arises by adding edges between all of the vertices of $\Z_n^d$ at distance at most $n^\zeta$ consists of $f$ components. Then 
\begin{align*}
\E{\prod_{u\in U} Y_u} \lesssim n^{-\alpha d f - \alpha d (1- p_d)/(1+p_d)+o(1)}, 
\end{align*}
where $p_d$ is as in Definition~\ref{def:pddef} and the constant in $\lesssim$ depends only on $c$ and $d$.
\end{lemma}

\begin{proof}
First we decompose $U$ into its $f$ connected components, i.e.\ every component contains points that are within distance $n^\zeta$ from some point of the same component. If two points belong to different components, then their distance is at least $n^\zeta$. Let $a$ be the number of components $(A_i)$ containing exactly one point and let $b$ be the number of components $(A'_i)$ containing at least two points. Since $U\notin \SS(\zeta)$, it follows that $b\geq 1$. For $i=1,\ldots, a$ we let $Y_{1,i}= \1(\tau_{a_i}>\tau_{S_{a_i}})$, where $A_i=\{a_i\}$. For $i=1,\ldots, b$ we pick $x_i, y_i \in A_i'$ distinct such that $\|x_i-y_i\|\leq n^\zeta$ and we set $Y_{2,i} = \1(\tau_{x_i},\tau_{y_i}>\tau_{S_{x_i}})$. Note that for $\zeta>0$ small enough $S_{x_i}=S_{y_i}$.
Let $k=\sum_{i=1}^{b}\1\left( \|x_i-y_i\|\leq n^{\zeta/(10d)} \right)$. 

For $j=1,\ldots, b$ we place two balls centered at each $x_j$ satisfying $\|x_j-y_j\|\leq n^{\zeta/(10d)}$ of radii~$n^{2\zeta/d}$ and~$n^\zeta/2$. For each $j$ not satisfying the above condition we place two balls around $x_j$ of radii $n^{\zeta/(15d)}$ and $n^{\zeta/(10d)}/2$. We also place two balls of the same radii around the corresponding $y_j$. As in Lemma~\ref{lem:n3} we denote by $N_u= N_u(n^{2\zeta/d},n^{\zeta}/2,\alpha t_*)$ and $N_u'=N_u(n^{\zeta/(15d)},n^{\zeta/(10d)}/2,\alpha t_*)$ for $u\in 
U$. By conditioning on the events $\{N_u>\underline{L}(\alpha t_*)\}$ and $\{N_u' >\underline{L}(\alpha t_*)\}$ depending on the radii of the balls that we placed around $u$ and using~\eqref{eq:xyafterR} in the case when $\|x_i-y_i\|\leq n^{\zeta/(10d)}$ we get exactly in the same way as in the proof of Lemma~\ref{lem:yhit} that 
\begin{align*}
\E{\prod_{u\in U}Y_u} \leq \E{\prod_{i=1}^{a} Y_{1,i}\cdot  \prod_{j=1}^{b}Y_{2,j}} \leq n^{-\alpha d a}\cdot n^{-2\alpha d k/(1+p_d)} \cdot n^{-2\alpha d (b-k)} \cdot n^{o(1)}.
\end{align*}
Since $k\leq b$, $a+b =f$, and $b\geq 1$ from the above we deduce
\begin{align*}
\E{\prod_{u\in U} Y_u}\leq n^{-\alpha d f -  \alpha d( 1- p_d)/(p_d+1) + o(1)},
\end{align*}
 and this finishes the proof.
\end{proof}

\begin{proposition}\label{pro:separated}
Fix $\alpha>(1+p_d)/2$,  $0<\gamma<2\alpha -1$  and let
\[
Z _\gamma= \sum_{x,y: \norm{x-y}\leq n^\gamma} \1(\tau_x >\tau_{S_x})\1(\tau_y>\tau_{S_y}).
\]
Then $\E{Z_\gamma} = o(1)$ as $n\to \infty$. 
\end{proposition}

\begin{remark}\rm{
We will show in the proof of the lower bound of Theorem~\ref{thm:uniform} that the threshold $(1+p_d)/2$ is sharp: for $\alpha\in (0,(1+p_d)/2)$ the random variable $Z_\gamma$ from the statement of Proposition~\ref{pro:separated} tends to~$\infty$ almost surely for any $\gamma>0$.
}
\end{remark}

\begin{proof}[Proof of Proposition~\ref{pro:separated}]

For $0<\zeta<\phi$ to be determined shortly we write 
\[
Z_\gamma= \sum_{x,y: \norm{x-y}\leq n^\zeta}\1(\tau_x>\tau_{S_x})\1(\tau_y >\tau_{S_y}) + \sum_{x,y: n^\zeta\leq \norm{x-y}\leq n^\gamma} \1(\tau_x>\tau_{S_x})\1(\tau_y >\tau_{S_y}).
\]
From Lemma~\ref{lem:2points} with $f=1$ we get
\begin{align*}
\E{\sum_{x,y: \norm{x-y}\leq n^\zeta}\1(\tau_x>\tau_{S_x}) \1(\tau_y>\tau_{S_y}) } \lesssim n^{d+d\zeta}\cdot n^{-2\alpha d/(p_d+1) + o(1)}.
\end{align*}
Hence for $\zeta<2\alpha/(p_d+1) - 1$ we get that the above upper bound is $o(1)$ as $n\to \infty$. From Lemma~\ref{lem:yhit} with $|U|=2$ we get
\begin{align*}
\E{\sum_{x,y: n^\zeta\leq \norm{x-y}\leq n^\gamma} \1(\tau_x >\tau_{S_x}) \1(\tau_{S_y}>\tau_{S_y})} \lesssim n^{d+d\gamma} \cdot n^{-2\alpha d(1+o(1))}.
\end{align*}
Therefore taking $\gamma<2\alpha-1$ we conclude that $\E{Z_\gamma} =o(1)$ as $n\to \infty$ and this completes the proof.
\end{proof}

\section{Total variation distance}\label{sec:totalvar}

In this section we give the proof of Theorem~\ref{thm:uniform}. As mentioned in Section~\ref{subsec:strategy} we will proceed by using the concentration estimates from Section~\ref{sec:excursions} to reduce the problem to proving the uniformity of the last visited set in each box in an appropriately chosen partition of~$\Z_n^d$. In order to establish the latter we will use the general strategy employed in the proof of~\cite[Theorem~6]{Prata_thesis}.

Let $t_*$ be as defined in~\eqref{eq:deftstar} in Section~\ref{sec:separated}. Let $Q=(Q_z)$ where $Q_z = \1(\tau_z > \alpha t_*)$ and $Z=(Z_z)$, where $Z_z$ are i.i.d.\ Bernoulli random variables of parameter $n^{-\alpha d}$. Recall the definition of~$\AA$, the process $Y$ and the collection of boxes $\SS_\beta$, where $\beta=\alpha-\epsilon$, defined in the setup subsection at the beginning of  Section~\ref{sec:separated} and in Definition~\ref{def:boxes}.
We define $\til{Q}$ by setting $\til{Q}_z = 0$ for all $z \in \AA$ and $\til{Q}_z = Q_z$ for $z\notin \AA$. We also define $\til{Z}$ by setting $\til{Z}_z = 0$ for $z \in \AA$ and $\til{Z}_z=Z_z$ for $z\notin \AA$.

\begin{claim}\label{cl:annulusx}
If $\alpha$, $\phi$ and $\epsilon$ satisfy
$d-(d+1)\alpha+\epsilon + \phi<0$, then we have as $n\to \infty$
\begin{align*}
\|\LL(Q) - \LL(\til{Q})\|_{\rm{TV}} = o(1) \quad \text{ and } \quad \|\LL(Z) -\LL(\til{Z})\|_{\rm{TV}} = o(1).
\end{align*}
\end{claim}

\begin{proof}
Using the obvious coupling between $Q$ and $\til{Q}$ we get
\begin{align*}
\|\LL(Q) - \LL(\til{Q})\|_{\rm{TV}} \leq \pr{\exists z\in \AA: Q_z=1 } \leq |\AA| \pr{\tau_z>\alpha t_*}.
\end{align*}
Since the volume of each annulus is of order $n^{(d-1)\beta+\phi}$ and the total number of annuli in the torus is of order $n^{d-d\beta}$, using Lemma~\ref{lem:hit} we get
\begin{align*}
\|\LL(Q) - \LL(\til{Q})\|_{\rm{TV}} \lesssim n^{d-d\beta}\cdot n^{(d-1)\beta+\phi} \cdot n^{-\alpha d} = n^{d-(d+1)\alpha + \epsilon + \phi} = o(1),
\end{align*}
where in the last step we used the assumption of the Claim. In exactly the same way we get the result for $Z$ and $\til{Z}$.
\end{proof}

\begin{lemma}\label{lem:sandwich}
We have 
\[
\pr{\til{Q}\neq Y} = o(1) \quad \text{as} \quad n \to \infty.
\]
\end{lemma}

We prove Lemma~\ref{lem:sandwich} at the end of this section.
We now proceed to the proof of Theorem~\ref{thm:uniform}.

\proof[Proof of Theorem~\ref{thm:uniform} Part I, existence of $\alpha_1(d)$]

Let $\alpha>(1+p_d)/2$.
The statement of the theorem is equivalent to showing
\[
\|\LL(Q) - \LL(Z)\|_{\rm{TV}} = o(1) \ \text{ as } \ n\to \infty.
\]
By the triangle inequality for total variation distance we have
\begin{align*}
\|\LL(Q) - \LL(Z)\|_{\rm{TV}}  &\leq \|\LL(Q) - \LL(\til{Q})\|_{\rm{TV}} +\|\LL(\til{Q}) - \LL(Y)\|_{\rm{TV}}\\
&+ \| \LL(Y) - \LL(\til{Z})   \|_{\rm{TV}}+\|\LL(\til{Z}) - \LL(Z)\|_{\rm{TV}}. 
\end{align*}
By Claim~\ref{cl:annulusx} and Lemma~\ref{lem:sandwich} it is enough to show that 
\begin{align*}
\|\LL(Y) - \LL(\til{Z})\|_{\rm{TV}}  = o(1) \text{ as } n\to \infty.
\end{align*}
Since $Y_z=\til{Z}_z=0$ for $z\in \AA$, in the total variation distance we only consider the distance between the law $\mu$ of $(Y_z)_{z\notin \AA}$ and the law $\nu$ of $(\til{Z}_z)_{z\notin \AA}$. 

For $\gamma =2\alpha-1-2\epsilon$ we define the collection of $n^\gamma$-separated subsets of $\Z_n^d\setminus\AA$ via
\[
\SS = \{ S\subseteq \Z_n^d\setminus \AA: \forall x,y \in S, \ \norm{x-y} \geq n^\gamma\}.
\]
For the total variation distance between $\mu$ and $\nu$ we have
\begin{align}\label{eq:totalvar}
\|\mu-\nu\|_{\rm{TV}} = \sum_{S\in \SS} (\mu(S) - \nu(S))_+ + \sum_{S\notin \SS} (\mu(S)-\nu(S))_+,
\end{align}
where abusing notation we write 
\[
\mu(S) = \pr{Y_z=1, z\in S, Y_u=0, u\notin \AA\cup S}.
\]
Let $Z_\gamma$ be as in Proposition~\ref{pro:separated}.
Since $(a-b)_+ \leq a$ for $a,b>0$, we can bound by Markov's inequality
\begin{align*}
\sum_{S\notin \SS} (\mu(S) - \nu(S))_+ \leq \sum_{S\notin \SS} \mu(S) \leq \E{Z_\gamma} = o(1),
\end{align*}
where the last equality follows from Proposition~\ref{pro:separated}, since $\gamma\in (0,2\alpha-1)$ and $\alpha>(1+p_d)/2$.
Let $M$ satisfy $d-\alpha d-\epsilon dM<0$.
For $B\in\SS_\beta$ we define the collections 
of sets 
\[
\SS_B = \{ S\in \SS: S\subseteq B\} \ \text{ and } \ 
\SS_M = \{ S\in\SS: |S\cap B|\leq M, \ \forall B\in \SS_\beta\}.
\]
Using again  $(a-b)_+\leq a$ for $a,b>0$
we now get
\begin{align*}
\sum_{S\in \SS} (\mu(S) - \nu(S))_+  &= \sum_{S \in \SS \setminus \SS_M} (\mu(S) - \nu(S))_+  + \sum_{S\in \SS_M} (\mu(S) - \nu(S))_+
\\
&\leq \sum_{ S\in \SS \setminus \SS_M} \mu(S) + \sum_{S\in \SS_M} (\mu(S) - \nu(S))_+.
\end{align*}
We now show that $\sum_{S \in \SS \setminus \SS_M} \mu(S) = o(1)$ as $n\to \infty$. Setting $\U= \{ x\notin \AA: \tau_x>\tau_{\SS_x}\}$ we get by the union bound 
\begin{align*}
\sum_{S\in \SS \setminus \SS_M}\mu(S) &= \pr{\U \in \SS\setminus \SS_M} = \pr{\exists B\in \SS_\beta, W\in \SS: |W|=M+1, W\subseteq \U\cap B} \\&\leq 
\sum_{\substack{B\in \SS_\beta, W \in \SS_B\\ |W|=M+1}}\pr{W\subseteq \U} \lesssim n^{d-d\beta} {n^{d\beta} \choose M+1} n^{-\alpha d(M+1)(1+o(1))} \\&\leq  
n^{d-d\beta} \frac{n^{d\beta(M+1)}}{(M+1)!} n^{-\alpha d(M+1)(1+o(1))} = \frac{1}{(M+1)!} n^{d- \epsilon d M -\alpha d+o(1)},
\end{align*}
where in the second inequality we used Lemma~\ref{lem:yhit}.
Since $d-\alpha d - \epsilon d M <0$ we obtain that 
\[
\sum_{S \in \SS \setminus \SS_M} \mu(S) = o(1) \ \text{ as } n \to \infty.
\]
Therefore we only need to show that 
\begin{align}\label{eq:goalx}
\sum_{S\in \SS_M} (\mu(S) - \nu(S))_+ = o(1) \ \text{ as } n\to \infty.
\end{align}
Let $\F$ denote the $\sigma$-algebra generated by 
$X(\tau_i(S))$ and $X(\sigma_i(\ol{S}))$ for all $S\in \SS_\beta$ and $i\geq 0$, where $\tau_i(S)$ and $\sigma_i(\ol{S})$ refer to the stopping times as defined at the beginning of Section~\ref{sec:excursions} with respect to the annulus~$\ol{S}\setminus S$. Then conditioning on $\F$, the collections $(Y_z)_{z\in B}$, for~$B\in \SS_\beta$ become independent. Therefore using the independence and Jensen's inequality, we have
\begin{align*}
&\sum_{S\in \SS_M} (\mu(S)-\nu(S))_+ = \sum_{S\in \SS_M} \left(\E{\prcond{Y_z=1,z\in S, Y_u=0, u\notin \AA\cup S}{\F}{}}  - \nu(S)  \right)_+.\\
&\leq  \sum_{S\in \SS_M} \E{\left(\prod_{B\in \SS_\beta}  \prcond{Y_z=1,z\in S\cap B, Y_u=0, u\in B\setminus(\AA\cup S)}{\F}{} - \prod_{B\in \SS_\beta}\nu(S\cap B) \right)_+  } \\
&\leq \sum_{B\in \SS_\beta} \sum_{\substack{S\in \SS_B\\ |S|\leq M}} \E{\left(\prcond{Y_z=1, z\in S, Y_u=0, u\in B\setminus(\AA\cup S)}{\F}{}  - \nu(S\cap B) \right)_+}.
\end{align*}
Around every $z\in \Z_n^d\setminus \AA$ we place two balls of radii~$r=n^{2\phi/\kappa}$ and~$R=n^\phi$ and we write $N_z$ for the number of excursions across the annulus $\B(z,R)\setminus \B(z,r)$ during the first 
$\underline{E}(\alpha t_*,\delta/4)$ excursions across $S_z\setminus\ol{S}_z$ as in Lemma~\ref{lem:n3}, where we recall~$\delta= n^{\phi(2-d)/\kappa+\psi}$ from~\eqref{eq:defdeltafinal}
and we take~$\psi>0$ very small. 
In some of the calculations below we have substituted the values of~$r$ and~$R$, except in a few places in order to emphasize the cancellation.
We set
\begin{align}\label{eq:elldef}
L=\frac{\alpha t_*}{(1+\delta)\trb_{r,R}} \quad \text{and} \quad L'= \frac{\alpha t_*}{(1-\delta)\trb_{r,R}}
\end{align}
and using Lemma~\ref{lem:n3} we get that there exists $C>0$ such that
$\sum_{S\in \SS_M} (\mu(S)-\nu(S))_+$ is upper bounded by
\begin{align}\label{eq:error}
\begin{split}
\sum_{B\in \SS_\beta} \sum_{\substack{S\in \SS_B\\ |S|\leq M}} \mathbb{E}\big[\big(\prcond{Y_z=1, z\in S, Y_u=0, u\in B\setminus(\AA\cup S), N_w\in (L,L'), w\in B}{\F}{}  \\- \nu(S\cap B) \big)_+ \big]
+n^C e^{-cn^\psi}.
\end{split}
\end{align} 
We now focus on the first term appearing in the expression above. We use the same technique as in the proof of~\cite[Theorem~6]{Prata_thesis}. By the inclusion-exclusion formula it is easy to see that 
\begin{align*}
\prcond{Y_z=1, z\in S, Y_u=0, u\in B\setminus(\AA\cup S), N_w\in (L,L'), w\in B}{\F}{} \\=
\sum_{\ell=0}^{n^{d\beta} -|S\cup (\AA\cap B)|} (-1)^{\ell} \sum_{W\in {B\setminus (S\cup \AA) \choose \ell}} \econd{\prod_{u\in S\cup W}Y_u\1(N_u\in (L,L')) }{\F} 
\end{align*}
and 
\begin{align*}
\nu(S\cap B) = \sum_{\ell=0}^{n^{d\beta} -|S\cup (\AA\cap B)|} (-1)^{\ell} \sum_{W\in {B\setminus (S\cup \AA) \choose \ell}} \left(\frac{1}{n^{\alpha d}} \right)^{|S|+|ell},
\end{align*}

where for a set $P$ and $\ell\in \N$ we write $P\choose \ell$ for the collection of subsets of~$P$ of size~$\ell$.
Let $K=1,\ldots, \frac{n^{d\beta} - |S\cup(\AA\cap B)|}{2}$ to be determined later. Applying the Bonferroni inequalities as in~\cite{ImbuzPrata, Prata_thesis} the sum in~\eqref{eq:error} is upper bounded by
\begin{align}\label{eq:superbig}
\nonumber \E{\sum_{B\in \SS_\beta}\sum_{\substack{S\in\SS_B\\ |S|\leq M}} \left[\sum_{\ell=0}^{2K}(-1)^\ell \sum_{W\in {B\setminus (S\cup \AA) \choose \ell}} \left(\econd{\prod_{v\in S\cup W } Y_v\1(N_v\in(L,L'))}{\F} -\left(\frac{1}{n^{\alpha d}} \right)^{|S|+\ell} \right)    \right]_+} \\
+ \sum_{B\in\SS_\beta} \sum_{\substack{S\in\SS_B\\ |S|\leq M}} \sum_{W\in {B\setminus (S\cup\AA) \choose 2K}} \left(\frac{1}{n^{\alpha d}}\right)^{|S|+2K}.
\end{align}
We start by showing that the second term in~\eqref{eq:superbig} is $o(1)$. Indeed, it can be bounded by 
\begin{align*}
\lesssim n^{d-d\beta} \sum_{s=0}^{M} {n^{d\beta} \choose s} {n^{d\beta} - s \choose 2K} \left(\frac{1}{n^{\alpha d}} \right)^{s+2K} \leq n^{d-d\beta} \sum_{s=0}^{M} \frac{n^{d\beta s}}{s!}\cdot \frac{n^{2d\beta K}}{(2K)!} \cdot \frac{1}{n^{\alpha ds + 2\alpha d K}} \\
= n^{d-d\beta}\sum_{s=0}^{M} \frac{1}{n^{d\epsilon s + 2 d\epsilon K}} \cdot \frac{1}{s! (2K)!}
\asymp \frac{1}{(2K)!} n^{d-\alpha d + d\epsilon - 2dK\epsilon}.
\end{align*}
Choosing $K>0$ such that $d-\alpha d + d\epsilon -2dK\epsilon<0$ gives that the above expression is $o(1)$. This leads us to choose~$K>\frac{1-\alpha +\epsilon}{2\epsilon}$.
Next we turn to bound the first term appearing in~\eqref{eq:superbig}.
To do that we split the sum over all $W\in {B\setminus (S\cup\AA) \choose \ell}$ into the sets $W$ such that $W\cup S \in \SS$ and into those $W$ such that $W \cup S \notin \SS$.  We also bound the positive part by the absolute value, so that we may forget about the term $(-1)^{\ell}$. Hence now we focus on proving that the following is~$o(1)$:
\begin{align}\label{eq:toobig}
\sum_{B\in\SS_\beta} \sum_{\substack{S\in\SS_B\\ |S|\leq M}} \sum_{\ell=0}^{2K}\sum_{\substack{W\in {B\setminus (S\cup \AA) \choose \ell}\\ W\cup S\in \SS}} \E{\left|\econd{\prod_{v\in S\cup W } Y_v\1(N_v\in(L,L'))}{\F} -\left(\frac{1}{n^{\alpha d}} \right)^{|S|+\ell} \right|} \\\label{eq:toobig2}+
\sum_{B\in\SS_\beta} \sum_{\substack{S\in\SS_B\\ |S|\leq M}} \sum_{\ell=1}^{2K}\sum_{\substack{W\in {B\setminus (S\cup \AA) \choose \ell}\\ W\cup S\notin \SS}} \E{\left|\econd{\prod_{v\in S\cup W } Y_v\1(N_v\in(L,L'))}{\F} -\left(\frac{1}{n^{\alpha d}} \right)^{|S|+\ell} \right|}
\end{align}

\begin{claim}\label{cl:one}
There exists $\alpha_1(d)\in(0,1)$ depending only on~$d$ such that for all $\alpha>\alpha_1(d)$ we have that the sum in~\eqref{eq:toobig} is~$o(1)$
as $n\to\infty$.
\end{claim}

\begin{proof}

Let $W\in {B\setminus S \choose \ell}$ such that $W\cup S \in \SS$.  Note that~$|W\cup S|=|S|+\ell$. 
Note that since $\gamma=2\alpha -1 -2\epsilon$, if we take $\phi$ satisfying the assumption of Claim~\ref{cl:annulusx} and $\epsilon>0$ sufficiently small, then $n^\phi<n^\gamma$. Hence we can use Lemma~\ref{lem:hittingprob} to get that almost surely
\begin{align*}
\left( 1 - \frac{\crw}{r^{d-2}}\left( 1 + O\left( \frac{1}{n^{\phi(\kappa-2)/\kappa}}\right)    \right)\right)^{L'(|S|+\ell)}
&\leq \econd{\prod_{v\in S\cup W } Y_v\1(N_v\in(L,L'))}{\F}\\
&\leq \left( 1 - \frac{\crw}{r^{d-2}}\left( 1 + O\left( \frac{1}{n^{\phi(\kappa-2)/\kappa}}\right)   \right)\right)^{L(|S|+\ell)}.
\end{align*}
Substituting the value of~$t_*$ into the expressions for~$L$ and~$L'$ from~\eqref{eq:elldef}, using Lemma~\ref{lem:hittingprob} and the value of $\delta$ (recall equation~\eqref{eq:defdeltafinal}) we get that
\begin{align}\label{eq:valuel}
L&= \frac{\alpha r^{d-2} \log(n^d)}{\crw} \left( 1+O\left( \frac{n^\psi}{n^{\phi(d-2)/\kappa}}\right)  \right) \quad \text{and}\\ 
\label{eq:valuel'}\ L'&= \frac{\alpha r^{d-2} \log(n^d)}{\crw}\left( 1+O\left( \frac{n^\psi}{n^{\phi(d-2)/\kappa}}\right)  \right). 
\end{align}
From~\eqref{eq:valuel} and using that for all $x$ we have $e^{-x} \geq 1 - x$ we get
\begin{align*}
&\left( 1 - \frac{\crw }{r^{d-2}}\left( 1 + O\left( \frac{1}{n^{\phi(\kappa-2)/\kappa}}\right)   \right)\right)^{L(|S|+\ell)} \\
&\leq \exp\left( - L(|S|+\ell)\frac{\crw}{ r^{d-2}}\left( 1+ O\left( \frac{1}{n^{\phi(\kappa-2)/\kappa}}\right) \right)     \right) 
\\&= \exp\left( - \alpha\log(n^d)(|S|+\ell)\left( 1+ O\left( \frac{n^\psi}{n^{\phi(\kappa-2)/\kappa}}\right) \right)  \right) \\
&= n^{-\alpha d(|S|+\ell)} \exp\left( -\alpha \log(n^d) (|S|+\ell) O\left( \frac{n^\psi}{n^{\phi(\kappa-2)/\kappa}}\right) \right) \\
&\leq n^{-\alpha d(|S|+\ell)} \left( 1 - \log(n^d)(|S|+\ell)O\left( \frac{n^\psi}{n^{\phi(\kappa-2)/\kappa}}\right)  \right),
\end{align*}
where in the last inequality we used that for all $x>0$ we have $e^{-x} \leq 1-x+x^2$ and that~$|S|+\ell$ is at most $M+2K$ which is independent of $n$.
Similarly substituting the value of $L'$ and using $1-x\geq e^{-x-2x^2}$ for $x\in(0,1/2)$ we obtain
\begin{align*}
&\left( 1 - \frac{\crw}{r^{d-2}}\left( 1 + O\left( \frac{1}{n^{\phi(\kappa-2)/\kappa}}\right) \right)\right)^{L'(|S|+\ell)}\\
&\geq \exp\left(-L'(|S|+\ell)\frac{\crw}{r^{d-2}} \left( 1 + O\left( \frac{1}{n^{\phi(\kappa-2)/\kappa}}\right) \right)  - L'(|S|+\ell)O\left(\frac{1}{r^{2(d-2)}}\right)\right) \\
&= \exp\left(  - \alpha\log(n^d)(|S|+\ell)\left( 1+ O\left( \frac{n^\psi}{n^{\phi(\kappa-2)/\kappa}}\right) \right) \right)\\
&= n^{-\alpha d(|S|+\ell)}\exp\left( -\log(n^d) (|S|+\ell)O\left( \frac{n^\psi}{n^{\phi(\kappa-2)/\kappa}}\right) \right)\\
&\geq  n^{-\alpha d(|S|+\ell)} \left( 1 -  \log(n^d)(|S|+\ell) O\left( \frac{n^\psi}{n^{\phi(\kappa-2)/\kappa}}\right) \right).
\end{align*}
Putting everything together we deduce
\begin{align*}
\left|\econd{\prod_{v\in S\cup W } Y_v\1(N_v\in(L,L'))}{\F} - \left(\frac{1}{n^{\alpha d}}\right)^{|S|+\ell} \right| \\
\leq n^{-\alpha d(|S|+\ell)} (|S|+\ell)O\left( \frac{n^\psi \log n}{n^{\phi(\kappa-2)/\kappa}}\right).
\end{align*}
Therefore the sum in~\eqref{eq:toobig} is bounded from above by 
\begin{align}\label{eq:aboveeq}
&\sum_{B\in\SS_\beta} \sum_{\substack{S\in\SS_B\\ |S|\leq M}} \sum_{\ell=0}^{2K}\sum_{\substack{W\in {B\setminus (S\cup \AA) \choose \ell}\\ W\cup S\in \SS}} n^{-\alpha d(|S|+\ell)} (|S|+\ell)O\left( \frac{n^\psi \log n}{n^{\phi(\kappa-2)/\kappa}}\right) \\
\nonumber &\leq \sum_{B\in\SS_\beta}\sum_{\substack{S\in\SS_B\\ |S|\leq M}} \sum_{\ell=0}^{2K}\frac{n^{d\beta \ell}}{\ell!} n^{-\alpha d (|S|+\ell)} (|S|+\ell)O\left( \frac{n^\psi \log n}{n^{\phi(\kappa-2)/\kappa}}\right) \\
\nonumber
&\leq n^{d-d\beta} \sum_{s=0}^{M}\frac{n^{d\beta s}}{s!} \sum_{\ell=0}^{2K}(M+2K)n^{-\alpha d s - \epsilon d \ell}O\left( \frac{n^\psi \log n}{n^{\phi(\kappa-2)/\kappa}}\right) \\
\nonumber&=n^{d-d\beta} \sum_{s=0}^{M}\frac{1}{s!} \sum_{\ell=0}^{2K}(M+2K)n^{-\epsilon d s - \epsilon d \ell}O\left( \frac{n^\psi \log n}{n^{\phi(\kappa-2)/\kappa}}\right)\\&\lesssim n^{d-d\beta - \frac{\phi}{\kappa}(\kappa-2)+ \psi}\log n.
\end{align}
Thus if 
\begin{align}\label{eq:condphi}
d-d\beta -\frac{\phi}{\kappa}(\kappa-2) +\psi<0,
\end{align}
then this last quantity is $o(1)$. Recall that $\phi$ was taken to satisfy $\phi<(d+1)\alpha -d - \epsilon$ from Claim~\ref{cl:annulusx}. These two inequalities together give that 
\[
\alpha>\frac{(\kappa-2)d + d\kappa}{(d+1)(\kappa-2)+ d\kappa} + \frac{\epsilon (\kappa-2) + d\epsilon \kappa + \psi \kappa}{(d+1)(\kappa-2)+d\kappa}.
\]
Since we can take~$\psi$ and~$\epsilon$ as small as we like, we deduce that for any
\begin{align}\label{eq:defalphaone}
\alpha>\frac{(\kappa-2)d + d\kappa}{(\kappa-2)(d+1)+ d\kappa} = :\alpha_1(d), 
\end{align}
the sum in~\eqref{eq:aboveeq} is $o(1)$ as $n\to\infty$ and this finishes the proof of the claim.
\end{proof}

\begin{remark}\label{rem:explanation}\rm{
We now explain how we chose the values of $r$, $R$, and $\delta$.  The error terms that come from the hitting estimate Lemma~\ref{lem:hittingprob} are $O(r/R)$ and $O(1/r^2)$ where $r$ and $R$ are the in and out radii, respectively, for the annulus that we put around each point. From the expressions \eqref{eq:valuel} and \eqref{eq:valuel'} for $L$ and $L'$, respectively, we get the additional factor of $1+O(\delta)$ where $\delta$ is as in \eqref{eq:defdeltafinal}. Combining the different estimates yields an error term which is of order $O(r/R)+O(1/r^2) +O(\delta)$.  From the concentration result (Lemma~\ref{lem:n3}) the smallest value of $\delta$ that we can choose is of order $r^{(2-d)/2} n^{\psi}$.  In particular, the value of $r$ essentially determines the value of $\delta$.  The largest value of $R$ that we can take is of order $n^\phi$ because we need the outer boundary of the annulus centred at a point $x \in \ul{S}$ for $S \in \SS_\beta$ to fit inside~$S$.  Given this choice, it is not hard to see that the optimal choice of $r$ is $\asymp n^{2\phi/\kappa}$.
}
\end{remark}

It only remains to show that the sum in~\eqref{eq:toobig2} is~$o(1)$.
This will follow from the following two claims:

\begin{claim}\label{cl:bernoulli}
If $\gamma \in (0,2\beta -1)$, then as $n\to \infty$
\[
\sum_{B\in\SS_\beta} \sum_{\substack{S\in\SS_B\\ |S|\leq M}} \sum_{\ell=1}^{2K}\sum_{\substack{W\in {B\setminus (S\cup \AA) \choose \ell}\\ W\cup S\notin \SS}} \left(\frac{1}{n^{\alpha d}} \right)^{|S|+\ell} = o(1).
\]
\end{claim}

\begin{proof}

Clearly we have
\begin{align}
\label{eq:sumsbig}
\nonumber\sum_{B\in\SS_\beta} \sum_{\substack{S\in\SS_B\\ |S|\leq M}} \sum_{\ell=1}^{2K}\sum_{\substack{W\in {B\setminus (S\cup \AA) \choose \ell}\\ W\cup S\notin \SS}} \left(\frac{1}{n^{\alpha d}} \right)^{|S|+\ell} &\lesssim n^{d-d\beta} \sum_{\substack{U\subseteq B:\\ U\notin\SS, |U|\leq M+2K}} \frac{1}{n^{\alpha d |U|}}\\
&= n^{d-d\beta}\sum_{m=2}^{M+2K} \sum_{\substack{U\subseteq B:\\U\notin\SS,|U|=m}} \frac{1}{n^{\alpha d m}}.
\end{align}
We now bound the total number of sets $U\subseteq B$ with $U\notin\SS$ such that $|U|=m$. Since $U\notin\SS$, there exist two points of $U$ that are at distance less than $n^\gamma$ from each other.  The number of ways of choosing these two points is $\lesssim n^{d\beta}\cdot n^{d\gamma}$. Then we have to pick another $m-2$ points. Therefore we get
\begin{align}
\label{eq:cardinalitysep}
|\{U\subseteq B: U\notin\SS, |U|=m\}| \lesssim n^{d\beta} \cdot n^{d\gamma}\cdot {n^{d\beta} \choose m-2} \leq n^{d\beta}\cdot n^{d\gamma} \cdot \frac{n^{d\beta (m-2)}}{(m-2)!}.
\end{align}
Hence~\eqref{eq:sumsbig} is 
\begin{align}
\lesssim n^{d-d\beta} \sum_{m=2}^{M+2K} n^{d\beta}\cdot n^{d\gamma} \cdot \frac{n^{d\beta (m-2)}}{(m-2)!}\cdot \frac{1}{n^{\alpha dm}} \lesssim n^{d-2d\beta + d\gamma - \epsilon d}. \label{eq:cardinalitysep2}
\end{align}
Since~$\gamma =2\alpha-1-2\epsilon$ we get that the expression in \eqref{eq:cardinalitysep2} is~$o(1)$ as $n\to \infty$.
\end{proof}

\begin{claim}\label{cl:two}
For all $\alpha>\alpha_1(d)$ we have as $n\to \infty$ that
\begin{equation}
\sum_{B\in\SS_\beta}\sum_{\substack{S\in \SS_B \\ |S|\leq M}}\sum_{\ell=1}^{2K}\sum_{\substack{W\in {B\setminus (S\cup \AA) \choose \ell} \\ W\cup S \notin \SS_B}} \E{\prod_{u\in S\cup W}Y_u\1(N_u\in (L,L'))} = o(1). \label{eq:cl_two_expression}
\end{equation}
\end{claim}

\begin{proof}
Fix $\zeta>0$; we will determine its precise value later.
First we define the collection of the $\zeta$-separated subsets of the box $B$ similar to Section~\ref{sec:separated}: $\SS_B(\zeta) = \{U\subseteq B: |x-y|\geq n^\zeta, \forall x,y \in U\}$. The expression in the left side of \eqref{eq:cl_two_expression} is upper bounded by
\begin{align*}
&\sum_{B\in\SS_\beta}\sum_{\substack{S\in \SS_B \\ |S|\leq M}}\sum_{\ell=1}^{2K}\sum_{\substack{W\in {B\setminus (S\cup \AA) \choose \ell} \\ W\cup S \notin \SS_B}} \E{\prod_{u\in S\cup W}Y_u} \leq \sum_{B\in\SS_\beta}\sum_{\substack{U\subseteq B: U\notin \SS_B,\\  |U|\leq M+2K}}\E{\prod_{u\in U}Y_u} \\
&= n^{d-d\beta}\sum_{\substack{U\subseteq B:U\notin \SS_B, \\ |U|\leq M+2K, U\in \SS_B(\zeta)}} \E{\prod_{u\in U}Y_u}  + n^{d-d\beta} \sum_{\substack{U\subseteq B:U\notin \SS_B, \\ |U|\leq M+2K, U\notin \SS_B(\zeta)}} \E{\prod_{u\in U}Y_u}= I + II.
\end{align*}
For the term $I$, using~\eqref{eq:cardinalitysep} and Lemma~\ref{lem:yhit}, since $U\in \SS_B(\zeta)$, we get
\begin{align*}
I\lesssim n^{d-d\beta}\sum_{m=2}^{M+2K}
n^{d\beta} \cdot n^{d\zeta} \cdot n^{d\beta(m-2)}\cdot \frac{1}{n^{\alpha dm(1+o(1))}} \asymp n^{d+d\zeta - 2d\beta -\epsilon d + o(1)}.
\end{align*}
If~$\zeta\in(0,2\alpha-1-\epsilon)$, this last quantity is $o(1)$. 
It remains to bound $II$. We view~$U\notin\SS_B(\zeta)$ with $U\subseteq B$ as a subset of the graph which arises by adding edges between all of the vertices of $\Z_n^d$ at distance at most $n^\zeta$. 
Writing $\SS(\zeta,f,m)$ for the collection of sets $U\subseteq B$ with $U\notin\SS_B(\zeta)$ and $|U|=m$ that consist of~$f$ components, we have
\[
|\SS(\zeta,f,m)| \lesssim n^{d\beta f} \cdot n^{d\zeta(m-f)}\cdot (m-1)^m,
\]
since first we choose one point for each component among the $n^{d\beta}$ possible points and then we connect the remaining~$m-f$ points to the already existing components. This upper bound and the same explanation appears in~\cite{Prata_thesis}. Using 
also Lemma~\ref{lem:2points} we deduce
\begin{align*}
II \leq n^{d-d\beta}\sum_{\substack{U\subseteq B: U\notin \SS_B,\\ U\notin \SS_B(\zeta),|U|\leq M+2K}} \E{\prod_{u\in U}Y_u} &\lesssim n^{d-d\beta} \sum_{m=2}^{M+2K}\sum_{f=1}^{m-1} \frac{n^{d\beta f} n^{d\zeta(m-f)}}{n^{\alpha df +\alpha d(1-p_d)/(1+p_d)+o(1)}}\\
&\asymp n^{d-2\alpha d/(1+p_d) +d\zeta(M+2K-1) + o(1)}.
\end{align*}
Since~for all~$d$ we have~$\alpha_1(d)>(1+p_d)/2$, by taking~$\zeta$ sufficiently small we see that this last quantity is $o(1)$ and this finishes the proof of the claim and the proof of the theorem.
\end{proof}

\begin{proof}[Proof of Lemma~\ref{lem:sandwich}]

We recall from~\eqref{eq:defdeltafinal} that~$\delta=n^{\phi(2-d)/\kappa+\psi}$ and recall from the setup in Section~\ref{sec:separated} that for~$S\in \SS_\beta$ we write $\tau_{S}$ for the first time that~$X$ has made $ \underline{E}(\alpha t_*, \delta/4)=
\alpha t_*\E{W}/((1+\delta/4)\trs_{n^\beta,10n^\beta})$~
excursions across the annulus $\ol{S}\setminus S$.

We now let
\[ \ul{\U}_{S} = \{z \in S : \tau_z \geq \tau_{S}\} \quad \text{and} \quad \ul{\U} = \bigcup_{S\in \SS_\beta} \ul{\U}_{S}.\]
Note that it suffices to show that $\pr{\U(\alpha t_*) = \ul{\U}}=1-o(1)$. 
If $x_{S}$ is the center of the box~$S\in \SS_\beta$, we write $N_{S}(t)= N_{x_{S}}(n^\beta,n^\beta+n^\phi,t)$.
Since the value of~$\delta$ satisfies the assumptions of Lemma~\ref{lem:n2} we immediately get
\begin{align}\label{eq:subset}
\pr{\U(\alpha t_*) \nsubseteq \ul{\U}} \leq \pr{\exists S\in \SS_\beta: N_{S}(\alpha t_*) <\underline{E}(\alpha t_*,\delta/4)} = o(1) \quad \text{as} \quad n\to \infty.
\end{align}
Therefore, it remains to show that $\pr{\ul{\U}\subseteq \U(\alpha t_*)}=1-o(1)$. 
We first note that 
\begin{align}\label{eq:taub}
\pr{\min_{S \in\SS_\beta} \tau_{S} \geq \alpha t_*(1-2 \delta)}=1-o(1).
\end{align}
Indeed, by Lemma~\ref{lem:n2} we have
\begin{align*}
&\pr{\min_{S \in\SS_\beta} \tau_{S} < \alpha t_*(1-2 \delta) } = \pr{\exists S\in \SS_\beta: \tau_{S}<\alpha t_*(1-2\delta)} \\
&\leq n^{d-d\beta} \pr{N_{S}(\alpha t_*(1-2\delta))>\underline{E}(\alpha t_*,\delta/4)} \lesssim n^{d-d\beta} \exp\left(-cn^{(d-2)(\beta-2\phi/\kappa)+\psi} \right) =o(1),
\end{align*}
since $2\phi/\kappa <\beta$ by Claim~\ref{cl:annulusx} provided that $\epsilon>0$ is sufficiently small.
For each box $S\in\SS_\beta$ and each point $z \in S$, let $\sigma_z$ be the first time that $X|_{[\tau_S,\infty)}$ has made 
\[ \frac{10 \delta t_*}{\trb_{n^\beta,10n^\beta}} \equiv E\]
excursions across the annulus $\B(z,10n^\beta)\setminus \B(z,n^\beta)$.
Then we have
\begin{align*}
\pr{\min_z \sigma_z <\alpha t_*} &= \pr{\min_z \sigma_z <\alpha t_*, \min_{S\in\SS_\beta} \tau_S>\alpha t_*(1-2\delta)} + o(1) \\
&\leq n^d \pr{\nbb_z(n^\beta,10n^\beta,2\alpha t_*\delta)>\frac{10t_*
\delta}{\trb_{n^\beta,10n^\beta}}} + o(1)=o(1),
\end{align*}
where the final assertion follows from Lemma~\ref{lem:n1}. (Lemma~\ref{lem:n1} is stated and proved for~$t\asymp n^d\log n$.  The same result and proof are also applicable for times $t>n^{3/2+\epsilon}$ for any fixed $\epsilon>0$. In this case the exponent in the first error term becomes $t/(\trb_{n^\beta,10n^\beta} n^\psi))$.) Consequently, 
\[
\pr{\min_z \sigma_z \geq \alpha t_*}=1-o(1),
\]
and hence it follows that 
\begin{align}\label{eq:first}
\pr{\{ z : \tau_z \geq \sigma_z\} \subseteq \U(\alpha t_*)}=1-o(1).
\end{align}
In order to show that $\pr{\ul{\U}\subseteq \U (\alpha t_*)}=1-o(1)$, it suffices to show that 
\[
\pr{ \{ z : \tau_z \geq \sigma_z\} = \ul{\U}}=1-o(1).
\]
By~\eqref{eq:subset} and~\eqref{eq:first} we only need to show that 
\begin{align}\label{eq:second}
\pr{\ul{\U}\subseteq \{z:\tau_z\geq \sigma_z\}}=1-o(1).
\end{align}
In order to prove this, we are going to get a bound on the probability that $X$ visits a given point~$z \in \ul{\U} \cap S$ in the time interval $[\tau_S,\sigma_z]$. By Lemma~\ref{lem:hittingprob} we obtain for constants $c_1, c_2, c_3>0$ that
\begin{align*}
\prcond{\tau_z\leq \sigma_z}{z\in \ul{\U}}{} = \prcond{z\text{ is hit in $[\tau_S,\sigma_z]$}}{z\in\ul{\U}}{} = \prcond{z\text{ is hit in $E$ excursions}}{z\in\ul{\U}}{} \\
\leq 1- \left(1-\frac{c_1}{n^{\beta(d-2)}}\right)^{E} \leq 1 - \exp\left(c_2n^{-\frac{\phi}{\kappa}(d-2)+\psi}\log n  \right) \leq  c_3 n^{-\frac{\phi}{\kappa}(d-2)+2\psi}. 
\end{align*}
We now use the above estimate to prove~\eqref{eq:second}. We have
\begin{align}\label{eq:sandwichone}
\begin{split}
\pr{\ul{\U}\nsubseteq \{z:\tau_z\geq \sigma_z\}} \leq \sum_z \pr{\tau_z\leq \sigma_z, z\in \ul{\U}} &=\sum_z \prcond{\tau_z\leq \sigma_z}{z\in \ul{\U}}{} \pr{z\in \ul{\U}} \\
&\lesssim\E{|\ul{\U}|} c_3n^{-\frac{\phi}{\kappa}(d-2)+2\psi}.
\end{split}
\end{align}
From Lemma~\ref{lem:yhit} we immediately get
\begin{align}\label{eq:sand2}
\E{|\ul{\U}|} = \sum_{z}\pr{Y_z = 1}  \lesssim n^{d-\alpha d}.
\end{align}
Therefore combining~\eqref{eq:sandwichone} and~\eqref{eq:sand2} we deduce
\begin{align*}
\pr{\ul{\U}\nsubseteq \{z:\tau_z\geq \sigma_z\}} \lesssim n^{d-\alpha d -\frac{\phi}{\kappa}(d-2) +2\psi},
\end{align*}
and using~\eqref{eq:condphi} it follows that for $\psi$ sufficiently small this last quantity is $o(1)$ as $n\to \infty$ and this concludes the proof. 
\end{proof}

\begin{proof}[Proof of Theorem~\ref{thm:uniform} Part II, existence of $\alpha_0(d)$]

We define
\begin{align}\label{eq:defwpartii}
W = \sum_{x,y: \|x-y\| =1} \1(Q_x =Q_y=1) \ \ \text{ and } \ \ U = \sum_{x,y:\|x-y\|=1}\1(Z_x=Z_y=1).
\end{align}
Since $\pr{Z_x = Z_y=1} = n^{-2\alpha d}$, we get that 
$\E{U} \asymp n^{d-2\alpha d}$.
Let $\epsilon\in(0,2\alpha p_d/(1+p_d))$. Then we have
\begin{align*}
\|\LL(Q) - \LL(Z)\|_{\rm{TV}} \geq \pr{W\geq n^{d-2\alpha d/(1+p_d) - \epsilon d}} -  \pr{U\geq n^{d-2\alpha d/(1+p_d) - \epsilon d}}.
\end{align*}
By Markov's inequality we immediately get
\[
\pr{U\geq n^{d-2\alpha d/(1+p_d) - \epsilon d}} = o(1) \ \text{ as } \ n\to \infty
\]
since $\epsilon<2\alpha p_d/(1+p_d)$. It thus remains to show that 
\begin{align}\label{eq:wgeq}
\pr{W\geq n^{d-2\alpha d/(1+p_d) - \epsilon d}} = 1-o(1) \ \text{ as } \ n\to \infty.
\end{align}
Let $L = \{(x_i,y_i)\}_{i=1}^{n^{d-\epsilon}}$ be a grid of points such that $\|x_i-y_i\|=1$ for all $i$ and $\|x_i-y_j\| \geq n^{\epsilon/d}$ for all~$i\neq j$.
We now place two balls around each pair of points~$x_i,y_i$ of radii $R=n^{\epsilon/d}/2$ and $r=n^{\epsilon/d^2}$. Let $N_i$ be the number of excursions in the annulus around the point~$x_i$ up to time~$\alpha t_*$. Let $E_i$ be the event that neither~$x_i$ nor~$y_i$ is covered during the $A'=\alpha t_*/((1-\delta)\trb_{r,R})$ excursions of the annulus around them, where $\delta=r^{(2-d)/2}n^{\psi}$ for some $\psi>0$ sufficiently small.  We now define
\[
\til{W} = \sum_{i=1}^{n^{d-\epsilon}} \1(E_i).
\]
Then by the union bound and Lemma~\ref{lem:n1} we have that
\[
\pr{\exists i: N_i>A'} = o(1) \ \text{ as } \ n \to \infty.
\] 
Therefore we get as $n\to \infty$ that
\[
\pr{W\leq \til{W}} \leq  \pr{\exists i:N_i>A'} = o(1).
\]
So we can now bound
\begin{align*}
\pr{W\geq n^{d-2\alpha d/(1+p_d) - \epsilon d}}
&\geq \pr{\til{W}\geq n^{d-2\alpha d/(1+p_d) - \epsilon d}, W\geq \til{W}}\\
&\geq \pr{\til{W}\geq n^{d-2\alpha d/(1+p_d) - \epsilon d}} + o(1).
\end{align*}
It thus suffices to show that 
\begin{align*}
\pr{\til{W}\geq n^{d-2\alpha d/(1+p_d) - \epsilon d}} = 1 - o(1) \ \text{ as } n\to \infty.
\end{align*}
Let $\F$ be the $\sigma$-algebra generated by $X(\tau_j(x_i))$ and $X(\sigma_j(x_i))$ for all $i$ and $j$, where $\tau_j(x_i)$ and $\sigma_j(x_i)$ are as defined at the beginning of Section~\ref{sec:excursions}. Then given $\F$ the events $E_i$ become independent. From~\eqref{eq:xyafterRneighbors} of Lemma~\ref{lem:2pointestimate} and using $1-x\geq e^{-x-2x^2}$ for $x\in(0,1/2)$ we get that for all~$i$ 
and all~$n$ sufficiently large
\begin{align*}
\prcond{E_i}{\F}{} \geq \left(1 - \frac{2\crw}{(1+p_d) r^{d-2}} + O\left(\frac{1}{r^d}\right) \right)^{A'} 
&\geq  \exp\left(-A'\frac{2\crw}{(1+p_d)r^{d-2}} +O\left(\frac{1}{r^{d}} \right)A' \right) \\ 
&\gtrsim n^{-2\alpha d/(1+p_d)+o(1)}.
\end{align*}

From the above it follows that for all~$n$ sufficiently large 
\begin{align*}
\econd{\til{W}}{\F} - n^{d-2\alpha d/(1+p_d) - \epsilon d} \geq \frac{\econd{\til{W}}{\F}}{2},
\end{align*}
and hence by Chebyshev's inequality we get
\begin{align*}
\pr{\til{W}\leq n^{d-2\alpha d/(1+p_d) - \epsilon d}} =\E{\prcond{\til{W}\leq n^{d-2\alpha d/(1+p_d) - \epsilon d}}{\F}{}} \\
\leq  \E{\prcond{\left|\til{W} - \econd{\til{W}}{\F}\right| \geq \frac{\econd{\til{W}}{\F}}{2}}{\F}{}} 
\leq 4\E{\frac{\var(\til{W}|\F)}{\econd{\til{W}}{\F}^2}}.
\end{align*}
Since conditional on $\F$ the events $E_i$ are independent, we get
\begin{align*}
\var(\til{W}|\F)= \sum_{i} \var(\1(E_i)|\F)
 = \sum_i\left(\prcond{E_i}{\F}{}  - \prcond{E_i}{\F}{}^2\right) \leq \econd{\til{W}}{\F}.
\end{align*}
Therefore, we deduce
\begin{align*}
\pr{\til{W}\leq n^{d-2\alpha d/(1+p_d) - \epsilon d}} \leq 4\E{\frac{1}{\econd{\til{W}}{\F}}} \lesssim \frac{1}{n^{d-2\alpha d/(1+p_d) -\epsilon+o(1)}}.
\end{align*}
Setting~$\alpha_0(d)=(1+p_d)/2$ gives that for all $\alpha\in(0,\alpha_0(d))$ if we take $\epsilon$ sufficiently small the quantity above is~$o(1)$ and this concludes the proof of the theorem.
\end{proof}

\section{Exact uniformity}
\label{sec:exact} In this section we prove Theorem~\ref{thm:exact}. We start with a preliminary lemma.

\begin{lemma}\label{lem:oneone}
Fix $\gamma>0$ and $\eta\in(0,1)$.
Let $A\subseteq \Z_n^d$ satisfy $A\in \SS(\gamma)$ {\rm{(}}recall~\eqref{eq:separdef}{\rm{)}}. Then for all~$x$ such that ${\rm{dist}}(x,A)\geq n^\gamma$ and all~$z\in A$ we have
\[
\prstart{X(\tau_{A}) =z}{x} = \frac{1}{|A|} + O_\eta(|A|n^{-\gamma(d-2)(1-\eta)}) \quad \text{as} \quad n\to \infty,
\]
where $\tau_{A}$ is the first hitting time of $A$ and $O_\eta$ means that the constants depend on~$\eta$. 
\end{lemma}

\begin{proof}
We let 
\[
\tunif = \min\left\{ t\geq 0: \max_{x,y}\left| 1- \frac{P^t(x,y)}{\pi(y)} \right| \leq \frac{1}{4} \right\}.
\]
Then it is standard that $\tunif \asymp c(d)n^2$ with $c(d)$ only depending on dimension. 
Let $\epsilon>0$ be sufficiently small.  We define
\[
\tau_A'=\inf\left\{t\geq n^\epsilon\tunif: X(t) \in A \right\}.
\]
Then we have
\begin{align}\label{eq:decprstart}
\begin{split}
&\prstart{X(\tau_A) =z}{x} =  \prstart{X(\tau_A) =z,\tau_A \geq n^\epsilon\tunif}{x}  + \prstart{X(\tau_A)  =z,\tau_A <n^\epsilon\tunif}{x} \\
&= \prstart{X(\tau_A')=z}{x} - \prstart{X(\tau_A')=z,\tau_A<n^\epsilon \tunif}{x} + \prstart{X(\tau_A) =z,\tau_A <n^\epsilon\tunif}{x}.
\end{split}
\end{align}
By the Markov property we have
\begin{align}\label{eq:pizrelate}
\begin{split}
\prstart{X(\tau_A') =z}{x} &= \sum_{y}\prstart{X(\tau_A')=z, X(n^\epsilon \tunif) =y}{x} \\
&= \sum_y \prstart{X(\tau_A)=z}{y} \prstart{X(n^\epsilon \tunif) = y}{x} \\
&=\left( 1 + O(e^{-cn^\epsilon})  \right) \prstart{X(\tau_A) =z}{\pi},
\end{split}
\end{align}
where the last equality follows from Proposition~\ref{prop:mixing_decay}.
Let $\tau_A^+$ be the first return time to $A$.
By reversibility we have for all $z\in A$
\begin{align}\label{eq:piexp}
\begin{split}
\prstart{X(\tau_A) =z}{\pi} &= \sum_{t\geq 0} \prstart{X(0) \notin A, X(1) \notin A,\ldots, X(t-1) \notin A, X(t) = z}{\pi} \\
&= \sum_{t\geq 0}  \prstart{X(0)=z, X(1)\notin A,\ldots, X(t-1) \notin A, X(t) \notin A}{\pi}\\
&= \sum_{t\geq 0} \pi(z) \prstart{X(1)\notin A,\ldots,X(t) \notin A}{z} \\&
= \sum_{t\geq 0} \pi(z) \prstart{\tau_A^+ >t}{z} = \pi(z) \estart{\tau_A^+}{z}.
\end{split}
\end{align}
Since $A \in \SS(\gamma)$, it follows that for all $w\in A$ we have~$A\cap \B(w) = \{w\}$, where $\B(w) = \B(w,n^\gamma/2)$.  
This now gives that for all~$w \in A$ 
\begin{align}\label{eq:same}
\estart{\tau_A^+\1(\tau_A^+<\tau_{\partial \B(w)})}{w} = K \quad \text{and} \quad \prstart{\tau_A^+>\tau_{\partial \B(w)}}{w} = s
\end{align}
where $K$ and $s$ are independent of $w$ and $\tau_{\partial\B(w)}$ is the first hitting time of $\partial\B(w)$. Therefore we get
\begin{align}\label{eq:plustaua}
\begin{split}
\estart{\tau_A^+}{z} &= \estart{\tau_A^+ \1(\tau_A^+ <\tau_{\partial \B(z)})}{z} + \estart{\tau_A^+\1(\tau_A^+ >\tau_{\partial \B(z)})}{z} \\
&= K + \estart{\tau_A^+\1(\tau_A^+ >\tau_{\partial \B(z)})}{z}. 
\end{split}
\end{align}
Using~\eqref{eq:same} we obtain for all $z\in A$
\begin{align}\label{eq:plusind}
\begin{split}
\estart{\tau_A^+\1(\tau_A^+ >\tau_{\partial \B(z)})}{z} &=\estart{ \estart{\tau_A}{X(\tau_{\partial \B(z)})}}{z} \prstart{\tau_A^+>\tau_{\partial \B(z)}}{z} \\
&= s\estart{ \estart{\tau_A}{X(\tau_{\partial \B(z)})}}{z}.
\end{split}
\end{align}
Writing for shorthand $\estart{F}{\partial\B} = \estart{ \estart{F}{X(\tau_{\partial \B(z)})}}{z} $
we deduce
\begin{align}\label{eq:partialbtaua}
\begin{split}
\estart{\tau_A}{\partial \B} & = \estart{\tau_A'\1(\tau_A\geq n^\epsilon \tunif)}{\partial \B} + O(n^{2+\epsilon})\\
&= \estart{\tau_A'}{\partial \B} - \estart{\tau_A'\1(\tau_A<n^\epsilon \tunif)}{\partial \B} + O(n^{2+\epsilon}).
\end{split}
\end{align}
Using again Proposition~\ref{prop:mixing_decay} as in the last step of \eqref{eq:pizrelate} we have
\begin{align}\label{eq:taua'pi}
\estart{\tau_A'}{\partial \B} = \left(1+O(e^{-cn^\epsilon})\right) \estart{\tau_A}{\pi}.
\end{align}
By H\"older's inequality for $p, q>1$ satisfying $1/p+1/q=1$ we get
\begin{align}\label{eq:expprob}
\estart{\tau_A'\1(\tau_A<n^\epsilon \tunif)}{\partial \B} \leq \estart{(\tau_A')^p}{\partial \B}^{1/p} \prstart{\tau_A<n^
\epsilon \tunif}{\partial \B}^{1/q}.
\end{align}
By the strong Markov property and Kac's moment formula \cite{FP_kac} we obtain
\begin{align*}
 \estart{(\tau_A')^p}{\partial \B}^{1/p} \lesssim \max_y \estart{n^{(2+\epsilon)p} + (\tau_A)^p}{y}^{1/p} \lesssim \left(n^{(2+\epsilon)p} + \max_{x,y}\estart{\tau_y}{x}^p\right)^{1/p} \lesssim n^d,
\end{align*}
since~$\max_{x,y}\estart{\tau_y}{x}\asymp n^d$ (this follows from instance from Lemma~\ref{lem:trRlem} for $r=1$). 
Writing $G(x,y) = \estart{\sum_{t=0}^{\tunif}\1(X(t)=y)}{x}$ for the Green kernel we have by Lemma~\ref{lem:green_kernel} that
\begin{align}\label{eq:seta}
 \prstart{\tau_A\leq \tunif}{\partial \B} \leq \sum_{w\in A} G(\partial \B(z),w) = O\left(|A|n^{-\gamma (d-2)}\right),
\end{align}
since ${\rm{dist}}(w,\partial \B(z)) \geq n^{\gamma}/2$ for all $w\in A$. By the union bound we get
\begin{align}\label{eq:aset}
\prstart{\tunif<\tau_A<n^
\epsilon \tunif}{\partial \B} \lesssim n^{2+\epsilon}\frac{|A|}{n^d} = O(|A|n^{2-d+\epsilon}).
\end{align}
Therefore, from~\eqref{eq:seta} and~\eqref{eq:aset} we deduce
\begin{align}\label{eq:abadd}
\begin{split}
\prstart{\tau_A<n^
\epsilon \tunif}{\partial \B} &= \prstart{\tau_A\leq \tunif }{\partial \B} + \prstart{\tunif<\tau_A<n^
\epsilon \tunif}{\partial \B} \\
&= O(|A| n^{-\gamma (d-2)}),
\end{split}
\end{align}
since~$\gamma\in(0,1)$ and $\epsilon>0$ is sufficiently small. Similarly we have
\begin{align}\label{eq:sameforx}
\prstart{\tau_A<n^\epsilon \tunif}{x} = O(|A| n^{-\gamma (d-2)}).
\end{align}
Substituting~\eqref{eq:abadd} into~\eqref{eq:expprob} we get
\begin{align*}
\estart{\tau_A'\1(\tau_A<n^\epsilon \tunif)}{\partial \B} \lesssim n^d\left(|A|n^{-\gamma(d-2)}\right)^{1/q}.
\end{align*}
Taking $1/q=1-\eta$ gives
\begin{align}\label{eq:tauataua'}
\estart{\tau_A'\1(\tau_A<n^\epsilon \tunif)}{\partial \B} \lesssim n^d\left(|A|n^{-\gamma(d-2)}\right)^{1-\eta}.
\end{align}
Plugging~\eqref{eq:taua'pi} and~\eqref{eq:tauataua'} into~\eqref{eq:partialbtaua} gives
\begin{align}\label{eq:exptaubpartial}
\begin{split}
\estart{\tau_A}{\partial B} = \left( 1 + O(e^{-cn^\epsilon})\right) \estart{\tau_A}{\pi} + O(n^d |A|^{1-\eta} n^{-\gamma(d-2)(1-\eta)}) +O(n^{2+\epsilon})\\
= \estart{\tau_A}{\pi} +  O(n^d |A|^{1-\eta} n^{-\gamma(d-2)(1-\eta)}) +O(n^{2+\epsilon}).
\end{split}
\end{align}
Combining~\eqref{eq:exptaubpartial} with~\eqref{eq:decprstart}, \eqref{eq:pizrelate}, 
\eqref{eq:piexp}, \eqref{eq:plustaua}, \eqref{eq:plusind} and~\eqref{eq:sameforx} results in
\begin{align*}
\prstart{X(\tau_A)=z}{x} = \frac{K+s\estart{\tau_A}{\pi}}{n^d}   +  O(|A| n^{-\gamma(d-2)(1-\eta)}).
\end{align*}
Since the first term appearing in the sum above is independent of $z$ by summing the above equality over all $z\in A$ we get
\begin{align*}
1 = |A| \left(\frac{K+s\estart{\tau_A}{\pi}}{n^d} \right) +O(|A|^{2} n^{-\gamma(d-2)(1-\eta)}).
\end{align*}
This implies that 
\begin{align*}
\frac{K+s\estart{\tau_A}{\pi}}{n^d} = \frac{1}{|A|} + O(|A|n^{-\gamma(d-2)(1-\eta)}).
\end{align*}
Finally we get
\begin{align*}
\prstart{X(\tau_A)=z}{x} = \frac{1}{|A|} + O(|A|n^{-\gamma(d-2)(1-\eta)})
\end{align*}
and this finishes the proof.
\end{proof}

\begin{proof}[Proof of Theorem~\ref{thm:exact} Part I, existence of $\alpha_1(d)$]

Let $t_1=(\alpha - \epsilon)t_*$, where $\alpha-\epsilon>\alpha_1(d)$ and $\alpha_1(d)$ is as in Theorem~\ref{thm:uniform}. For each $x\in \Z_n^d$ we let $Z_x=1$ with probability $n^{-d(\alpha - \epsilon)}$ and $0$ otherwise, independently over different $x\in \Z_n^d$. We set $V=\{x\in \Z_n^d: Z_x=1\}$. Then by Theorem~\ref{thm:uniform} we have that 
\begin{align*}
\|\LL(\U(t_1)) - \LL(V)\|_{\rm{TV}} = o(1) \quad \text{as} \quad n\to \infty,
\end{align*}
where we recall that $\U(t)$ is the uncovered set at time~$t$.
Therefore there exists a coupling of $V$ and $\U(t_1)$ such that 
\begin{align}\label{eq:success}
\pr{\U(t_1) \neq V} =o(1) \quad \text{as} \quad n\to \infty.
\end{align}
We now describe a coupling of the laws of~$\U(\tau_\alpha)$ and~$\W_\alpha$: First we fix $\gamma\in(0,2(\alpha -\epsilon) -1)$. We couple $\U(t_1)$ and $V$ using the optimal coupling. If $|V|<n^{d-\alpha d}$ or $V\notin \SS(\gamma)$, then we generate $\U(\tau_\alpha)$ and $\W_\alpha$ independently. If $|V| \geq n^{d-\alpha d}$ and $V\in \SS(\gamma)$, then we keep running the random walk until it has visited $n^d-n^{d-\alpha d}$ points. We also remove points from $V$ independently at random until we are left with a set on~$n^{d-\alpha d}$ points. Note that the resulting set is equal in distribution to $\W_\alpha$.

Let $\xi_1,\ldots,\xi_{|V|-n^{d-\alpha d}} \in \U(t_1)$ be the first $|V|-n^{d-\alpha d}$ points in $V$ visited by the random walk after time~$t_1$. Let $\zeta_1$ be uniform in~$V$. For each $2\leq j\leq |V|-n^{d-\alpha d}$ we inductively let~$\zeta_j$ be uniform in $V\setminus \{\zeta_1,\ldots, \zeta_{j-1}\}$. Then by Lemma~\ref{lem:oneone} there exists a coupling of $(\xi_i)$ and $(\zeta_i)$ such that
\begin{align}\label{eq:beforeunion}
\nonumber\prcond{\xi_i\neq \zeta_i}{V=\U(t_1)\in \SS(\gamma), \xi_j=\zeta_j, \forall j<i, |V|\leq n^{d-\alpha d+\epsilon}}{}
\\
\lesssim n^{-\gamma(d-2)(1-\epsilon)}\cdot n^{2(d-\alpha d+\epsilon)}.
\end{align}
We first couple $\xi_1$ and $\zeta_1$ using the above coupling. If this succeeds, then we couple $\xi_2$ and~$\zeta_2$ in the same way. If at some point the coupling fails, then we let the two processes evolve independently. Therefore we get
\begin{align}\label{eq:coupleunion}
\pr{\U(\tau_\alpha)\neq \W_\alpha} \leq \pr{\U(t_1)\neq V} +  \pr{|V| >n^{d-\alpha d+\epsilon}} + \pr{V\notin \SS(\gamma)} \\
+ \prcond{\exists i\leq |V|: \xi_i\neq \zeta_i}{\U(t_1)=V,|V|\leq n^{d-\alpha d+\epsilon},V\in \SS(\gamma)}{}.
\end{align}
Since $\E{|V|} = n^{d-\alpha d}$ by Markov's inequality we get as $n\to\infty$
\begin{align}\label{eq:normv}
\pr{|V|>n^{d-\alpha d+\epsilon}} = o(1).
\end{align}
Using Lemma~\ref{pro:separated} and~\eqref{eq:success} or by a straightforward calculation
we obtain  that for $\gamma\in(0,2(\alpha-\epsilon)-1)$
\begin{align}\label{eq:ssgamma}
\pr{V\notin\SS(\gamma)} \leq \pr{\U(t_1)\notin \SS(\gamma)} + o(1) = o(1).
\end{align}
By the union bound we now have
\begin{align*}
\prcond{\exists i\leq |V|: \xi_i\neq \zeta_i}{\U(t_1)=V,|V|\leq n^{d-\alpha d + \epsilon},V\in \SS(\gamma)}{} \leq n^{-\gamma(d-2)(1-\epsilon)}\left(n^{d-\alpha d+\epsilon}\right)^3 \\
= n^{-\gamma(d-2)(1-\epsilon) + 3d-3\alpha d+3\epsilon}.
\end{align*}
Using the expression for $\alpha_1(d)$ given in~\eqref{eq:defalphaone}, choosing $\epsilon$ sufficiently small and taking $\gamma = 2(\alpha - \epsilon) -1 - \epsilon$ give that 
the above quantity is~$o(1)$, since $\alpha-\epsilon>\alpha_1(d)$. This together with~\eqref{eq:coupleunion}, \eqref{eq:normv} and~\eqref{eq:ssgamma} implies that
 \begin{align*}
 \pr{\U(\tau_\alpha) \neq V} = o(1) \quad \text{as} \quad n\to \infty
 \end{align*}
and this concludes the proof.
\end{proof}

\begin{proof}[Proof of Theorem~\ref{thm:exact} Part II, existence of $\alpha_0(d)$]

The proof of this part follows in the same way as the proof of the existence of $\alpha_0(d)$ in Theorem~\ref{thm:uniform}. Let $\alpha_0(d)$ be as in Theorem~\ref{thm:uniform} and $\alpha>0$ with $\alpha+\epsilon<\alpha_0(d)$ with $\epsilon>0$ sufficiently small.
We let $Q_u=1(u\in \U(\tau_\alpha))$ and $Z_u=1(u\in \W_\alpha)$. Then  we define
\begin{align*}
W'=\sum_{x,y:\|x-y\|=1} \1(Q_x=Q_y=1) \quad \text{and} \quad U'=\sum_{x,y:\|x-y\|=1} \1(Z_x=Z_y=1).
\end{align*}
Then for all $x,y \in \Z_n^d$ distinct we have
\begin{align*}
\pr{Z_x=Z_y=1} = \frac{n^{d-\alpha d}}{n^d}\cdot \frac{n^{d-\alpha d}-1}{n^d-1},
\end{align*}
and hence $\E{U'} \asymp n^{d-2\alpha d}$. Let $t_1=(\alpha+\epsilon)t_*$. Then on the event $\{\tau_\alpha\leq t_1\}$ we have $W'\geq W$, where $W$ is defined in~\eqref{eq:defwpartii} in the proof of Theorem~\ref{thm:uniform} Part II. Take $\epsilon\in(0,2\alpha p_d/(1+p_d))$. Then we have
\begin{align*}
\|\LL(\U(\tau_\alpha)) - \LL(\W_\alpha)\|_{\rm{TV}} \geq \pr{W'\geq n^{d-2\alpha d/(1+p_d) -\epsilon d}} - \pr{U'\geq n^{d-2\alpha d/(1+p_d) -\epsilon d}}.
\end{align*}
By Markov's inequality we get
\begin{align*}
\pr{U'\geq n^{d-2\alpha d/(1+p_d) -\epsilon d}} = o(1) \quad \text{as} \quad n\to \infty,
\end{align*}
since $\epsilon\in(0,2\alpha p_d/(1+p_d))$. By Markov's inequality again we have
\begin{align*}
\pr{\tau_\alpha >t_1} = \pr{|\U(t_1)|>n^{d-\alpha d}} \leq \frac{1}{n^{\epsilon d}} = o(1) \quad \text{as} \quad  n\to \infty,
\end{align*}
where we used that $\E{|\U(t_1)|} \asymp n^{d-d(\alpha+\epsilon)}$. Therefore we get
\begin{align*}
\pr{W'\geq n^{d-2\alpha d/(1+p_d) -\epsilon d}}  \geq \pr{W'\geq n^{d-2\alpha d/(1+p_d) -\epsilon d}, \tau_\alpha<t_1 }\\
 \geq \pr{W\geq n^{d-2\alpha d/(1+p_d) -d\epsilon}} -o(1)  = 1 - o(1),
\end{align*}
where the last equality follows from~\eqref{eq:wgeq} in the proof of Theorem~\ref{thm:uniform} Part II and this concludes the proof.
\end{proof}

\section{Further questions}
\label{sec:questions}

$\ $

Throughout, we let $\ol{\alpha}_0(d)$ (resp.\ $\ol{\alpha}_1(d)$) be the largest (resp.\ smallest) value such that the assertions of~\eqref{eq:unif_a0}--\eqref{eq:exact_a1} hold.

\noindent{\bf Question 1.}
What are the precise values of $\ol{\alpha}_0(d)$ and $\ol{\alpha}_1(d)$?  Is it true that $\ol{\alpha}_0(d)$ corresponds to the threshold $\alpha_0(d)=(1+p_d)/2$ above which $\U(t)$ with high probability does not have neighbouring points while below which it does (as shown in Sections~\ref{sec:separated} and~\ref{sec:totalvar})?  Is there a phase transition: is it true that $\ol{\alpha}_0(d) = \ol{\alpha}_1(d)$?  Our lower bound $\alpha_0(d)$ for $\ol{\alpha}_0(d)$ converges to $\tfrac{1}{2}$ as $d \to \infty$.  Is this the correct asymptotic value of both $\ol{\alpha}_0(d)$ and $\ol{\alpha}_1(d)$ in the $d \to \infty$ limit (in agreement with the threshold for non-uniformity in the sense of \cite{MP_unif})?

\smallskip

\noindent{\bf Question 2.}  What is the asymptotic law of $\U(\alpha t_*)$ for $\alpha \in (0,\ol{\alpha}_0(d))$?  We proved in Theorem~\ref{thm:uniform} that $\U(\alpha t_*)$ for $\alpha \in (0,\alpha_0(d))$ is not uniformly random by showing that it contains more neighbours than a random subset of $\Z_n^d$ where points are included independently with probability $n^{-\alpha d}$.  The arguments of Section~\ref{sec:separated} generalize to give that for any $\alpha \in (0,1)$ there exists $k = k(\alpha)$ and $\gamma > 0$ such that each ball of radius~$n^\gamma$ contains at most~$k$ points with high probability.  This suggests that there is a way to describe $\U(\alpha t_*)$ by:
\begin{enumerate}[(i)]
\item sampling points in $\Z_n^d$ independently with probability $\asymp n^{-\alpha d}$ and then
\item decorating the neighbourhood of each such point in a given way.
\end{enumerate}

\smallskip

\noindent{\bf Question 3.}  For what class of graphs beyond $\Z_n^d$ for $d \geq 3$ do the results of Theorems~\ref{thm:uniform} and~\ref{thm:exact} also hold?

%

\appendix

\section{Elementary estimates}
\label{app:elementary}

We begin by recording a few elementary estimates for Markov chains and random walks.  Afterwards, we will give the proofs of several results stated in the text.  The following is a restatement of \cite[Proposition~3.3]{MP_unif}.

\begin{proposition}
\label{prop:mixing_decay}
Suppose that $p^s(\cdot,\cdot)$ denotes the transition kernel for a time-homogeneous Markov chain on a countable state space with a unique stationary distribution $\pi$.  For every $s,t \in \N$,
\begin{align}
 &\max_x \| p^{t+s}(x,\cdot) - \pi \|_{TV} \leq 4 \max_{x,y} \|p^t(x,\cdot) - \pi \|_{TV} \| p^s(y,\cdot) - \pi \|_{TV} \label{eqn:tv_decay}\\
 &\max_{x,y} \left| \frac{p^{t+s}(x,y)}{\pi(y)} - 1\right| \leq \max_{x,y} \frac{p^{s}(x,y)}{\pi(y)} \max_x \| p^{t}(x,\cdot) - \pi \|_{TV} \label{eqn:uniform_decay}.\
 \end{align}
\end{proposition}

It is easy to see that the following result can be derived from \cite[Theorem~4.3.1]{LawLim}.

\begin{lemma}
\label{lem:green_kernel}
Let $G(x,y) = \estart{\sum_{t=0}^{\tunif}\1(X(t)=y)}{x}$, where $\tunif$ is the uniform mixing time of random walk on~$\Z_n^d$.  There exist constants $c_1,c_2 > 0$ depending only on $d \geq 3$ such that
\[ c_1 |x-y|^{2-d} \leq G(x,y) \leq c_2 |x-y|^{2-d} \quad\text{for all}\quad x,y \in \Z_n^d.\]
\end{lemma}

The following is a standard hitting time estimate for random walk.

\begin{lemma}\label{lem:trRlem}
For all $r<n/4$ we have
\[
\max_{x\in\Z_n^d} \estart{\tau_{\partial\B(0,r)}}{x} \asymp \frac{n^d}{r^{d-2}}.
\]
\end{lemma}

\begin{proof}[Proof of Lemma~\ref{lem:coupling}]
It is clear that the sequence of exit points is a Markov chain. Since it is irreducible on a finite state space, it has a unique invariant distribution $\til{\pi}$. 

Fix $y\in \partial\ball{R}$.
We let $f(x) = \prcond{X(\tau_R) = y}{X(\tau_r)=x}{}$. Then $f$ is a harmonic function and since $r/R\geq 10$, then $\ball{r}$ or $\SS(0,r)$ are separated from $\partial \ball{R}$, so we can apply Harnack's inequality (Lemma~\ref{lem:harnackineq}) and thus we get a constant $c\geq 1$ such that for all $x,z\in \partial \ball{r}$ or $x,z \in \partial \SS(0,r)$ we have
\[
\frac{1}{c}f(z) \leq f(x) \leq c f(z)
\]
uniformly over all $y\in\partial\ball{R}$.
From that it follows that if $\nu_x$ is the law of $Y_j$ given that $Y_{j-1}=x$, then for all $x,z\in \partial \ball{R}$
\[
\frac{1}{c}\nu_z \leq \nu_x\leq c \nu_z.
\]
By using the optimal coupling between $\nu_x$ and $\nu_y$ we get that for all $x,y$
\[
\|\nu_x-\nu_y\|_{\rm{TV}} =1 -  \sum_{z}\nu_x(z)\wedge \nu_y(z) \leq 1-\frac{1}{c}\sum_{z}\nu_x(z) = 1-\frac{1}{c}.
\]
Therefore, since $\bar{d}(t)$ (defined in~\cite[Section~4.4]{LPW})) is sub-multiplicative, we get that for all $t$
\[
\bar{d}(t) \leq \left(1-\frac{1}{c} \right)^t.
\]
This now immediately gives that $\tmix=k_0<\infty$ and independent of the size of the state space. 

Let $\mu$ denote the law of $(Y_N,\ldots,Y_{mN})$, then we have
\begin{align*}
\|\mu-\til{\pi}^{\otimes m}\|_{\rm{TV}} &= \frac{1}{2}\sum_{y_1,\ldots,y_m}|\mu(y_1,\ldots,y_m) - \til{\pi}(y_1)\cdots\til{\pi}(y_m)| \\
&= \frac{1}{2}\sum_{y_1,\ldots,y_m}|\mu(y_m|y_1,\ldots,y_{m-1})\cdots \mu(y_2|y_1)\mu(y_1) - \til{\pi}(y_1)\cdots\til{\pi}(y_m)|,
\end{align*}
where we write $\mu(y_j|y_1,\ldots,y_{j-1})$ for the conditional probability that $Y_{jN}=y_j$ given $Y_{iN}=y_i$ for all $1\leq i\leq j-1$.
Using~Proposition~\ref{prop:mixing_decay} we get 
\[
\frac{\mu(y_{j}|y_1,\ldots,y_{j-1})}{\til{\pi}(y_j)} = 1+O(e^{-cN}).
\]
Substituting in the formula above we get for $me^{-N}<1$
\begin{align*}
\|\mu-\til{\pi}^{\otimes m}\|_{\rm{TV}} = \frac{1}{2}\sum_{y_1,\ldots,y_m}\til{\pi}(y_1)\cdots \til{\pi}(y_m)\left| (1+O(e^{-cN}))^m - 1 \right| \lesssim me^{-cN},
\end{align*}
where in the last step we used that
\[
e^x-1\leq 10x \ \text{ for } \ x<1
\]
and this completes the proof of the lemma.
\end{proof}

\begin{proof}[Proof of Claim~\ref{cl:geom}]

Since $p\in(0, 1/2]$, we have that $p/(1-p)\leq 1$, and hence
\[
\E{X^j} = \sum_{x=1}^{\infty} x^j(1-p)^{x-1}p \leq  \sum_{x=1}^{\infty} x^je^{-px}.
\]
We are now going to compare the sum appearing on the right hand side above to the integral $\int_{1}^{\infty} x^je^{-px}\,dx$. The function $f(x) = x^j e^{-px}$ is increasing for $x\leq j/p$ and decreasing for $x>j/p$. We thus have
\begin{align*}
\int_{1}^{[j/p]} x^j e^{-xp}\,dx = \sum_{k=1}^{[j/p]-1}\int_{k}^{k+1}x^je^{-xp}\,dx 
\geq \sum_{k=1}^{[j/p]-1} k^je^{-kp}
\end{align*}
and 
\begin{align*}
&\int_{[j/p]+1}^{\infty} x^j e^{-xp}\,dx = \sum_{k=[j/p]+1}^{\infty}\int_{k}^{k+1} x^je^{-xp}\,dx\\
\geq& \sum_{k=[j/p]+1}^{\infty} (k+1)^je^{-(k+1)p} 
= \sum_{k=[j/p]+2}^{\infty}k^je^{-kp}.
\end{align*}
Therefore we get
\begin{align*}
\sum_{k=1}^{\infty} k^je^{-kp} =  \sum_{k=1}^{[j/p]-1} k^je^{-kp} + f([j/p]) + f([j/p]+1) +  \sum_{k=[j/p]+2}^{\infty}k^je^{-kp}.
\end{align*}
Since the function $f$ achieves its maximum at $j/p$ we have that $f(x) \leq (j/p)^je^{-j}$ for all $x$. Using the above inequalities we get
\begin{align*}
\sum_{k=1}^{\infty} k^je^{-kp} \leq 2(j/p)^je^{-j} + \int_{1}^{\infty} x^je^{-xp}\,dx.
\end{align*}
It is easy to see that the integral appearing above is equal to $j!/p^j$ (it is the Gamma function), and using Stirling's formula we get
\[
\sum_{k=1}^{\infty} k^je^{-kp} \lesssim \frac{j!}{p^j}
\]
and this finishes the proof of the claim.
\end{proof}

\begin{proof}[Proof of Lemma~\ref{lem:harnackineq}]

By~\cite[Theorem~6.3.8, equation~(6.19)]{LawLim} and using the fact that $R>2r$ we get that there exists a universal constant $c_1$ such that for all $u,v \in \B(0,r)$ with $\|u-v\|=1$ 
\begin{align}\label{eq:diff}
|f(u)-f(v)| \leq c_1 \frac{f(u)}{R}.
\end{align}
By Harnack's inequality (see for instance~\cite[Theorem~6.3.9]{LawLim}) we get for a universal constant~$c_2$ that
\begin{align}\label{eq:harn}
\max_{u\in \B(0,r)}f(u) \leq c_2 f(y).
\end{align}
Let $u_0=x,u_1,\ldots,u_{\ell-1}, u_\ell=y$ be the shortest path from $x$ to $y$ such that $\|u_{i+1} -u_i\|=1$ for all~$i$. Notice that the assumption~$x,y\in \B(0,r)$ gives that $\ell\leq 2r$ and $u_i\in \ball{r}$ for all~$i$.
We thus obtain
\begin{align*}
\left|\frac{f(x)}{f(y)} -1
\right| = \frac{|f(x) - f(y)|}{f(y)}  \leq \sum_{i=0}^{\ell-1}\frac{|f(u_{i+1}) - f(u_i)|}{f(y)} \leq \sum_{i=0}^{\ell-1}\frac{c_1 f(u_i)}{Rf(y)}  \leq \frac{2c_1c_2r}{R},
\end{align*}
where in the second inequality we used~\eqref{eq:diff} and for the last one we used~\eqref{eq:harn}. Therefore we deduce
\begin{align}\label{eq:rR}
\frac{f(x)}{f(y)} = 1+O\left(\frac{r}{R} \right)
\end{align}
and this concludes the proof.
\end{proof}

\begin{proof}[Proof of Lemma~\ref{lem:hittingprob}]

Let $G$ be the Green kernel for simple random walk in $\Z^d$. Then by~\cite[Theorem~4.3.1]{LawLim} we have that as $\|x\|\to \infty$, then
\[
G(x) = \frac{c_d}{\|x\|^{d-2}} + O\left(\frac{1}{\|x\|^d}\right),
\]
where $c_d$ is a constant that only depends on the dimension~$d$.
By Bayes' formula we have
\begin{align}\label{eq:bayes}
\prcond{\tau_z<\tau_R}{X(\tau_R)=y}{x} = \frac{\prcond{X(\tau_{R})=y}{\tau_z<\tau_R}{x}\prstart{\tau_z<\tau_R}{x}}{\prstart{X(\tau_R)=y}{x}}.
\end{align}
We now treat the term $\prstart{\tau_z<\tau_R}{x}$ and the ratio $\prcond{X(\tau_{R})=y}{\tau_z<\tau_R}{x}/\prstart{X(\tau_R)=y}{x}$ separately.
By transitivity in expressions involving the Green kernel we will take~$z=0$. However, $\norm{z}$ refers to the setting without the translation. Since the Green kernel is harmonic outside of $0$, we can apply the optional stopping theorem to get
\begin{align*}
G(x) = G(0) \prstart{\tau_z<\tau_R}{x} + \escond{G(X(\tau_R))}{\tau_R<\tau_z}{x} \left(1- \prstart{\tau_z<\tau_R}{x}\right).
\end{align*}
Since $R-\|z\|\leq \|X(\tau_R)\| \leq R+\|z\|$ and $r-\|z\|\leq \|x\| \leq r+\|z\|$ and we have that $\|z\|\leq r/4$, $r, R\to \infty$ as $n\to \infty$ by substituting in the asymptotic expression for the Green kernel, we get
\begin{align}\label{eq:hitz}
\prstart{\tau_z<\tau_R}{x} = \frac{c_d}{G(0)r^{d-2}} \left( 1+O\left(\left(\frac{r}{R}\right)^{d-2}\right)+O\left(\frac{1}{r^2} \right) +O\left(\frac{\norm{z}}{r} \right) \right).
\end{align}
Now it remains to bound the ratio 
\begin{align*}
\frac{\prcond{X(\tau_R)=y}{\tau_z<\tau_R}{x}}{\prstart{X(\tau_R)=y}{x}} = \frac{\prstart{X(\tau_R)=y}{z}}{\prstart{X(\tau_R)=y}{x}},
\end{align*}
where the equality follows by the strong Markov property. If we set $f(w) = \prstart{X(\tau_R)=y}{w}$, then it is easy to check that $f$ is harmonic in $\ball{R}$. 
Since by assumption $x,z\in \B(0,r)$ Lemma~\ref{lem:harnackineq} gives
\begin{align}\label{eq:rR}
\frac{f(z)}{f(x)} = 1+O\left(\frac{r}{R} \right)
\end{align}
Plugging~\eqref{eq:hitz} and~\eqref{eq:rR} into~\eqref{eq:bayes} and setting $\crw=c_d/G(0)$ 
gives
\begin{align*}
\prcond{\tau_z<\tau_R}{X(\tau_R)=y}{x} = \frac{\crw}{r^{d-2}}\left(1+O\left(\frac{r}{R} \right) +O\left(\frac{1}{r^2} \right)+O\left(\frac{\norm{z}}{r}
\right)\right) 
\end{align*}
and this concludes the proof.
\end{proof}

\section{Proof of Lemma~4.1}
\label{app:hitting}

We start with some preliminary results. Throughout we assume that $R=o(n)$ and $R\geq 2r$.
First we let $\tau=\sigma_1-\sigma_0$, where the $\sigma_i$'s are defined in Section~\ref{sec:excursions} and we take $F(x,R)= \B(0,R)$ and $E(x,r)= \B(0,r)$. 
We start by proving that up to small error the expectation of $\tau$ does not depend on the starting point of $X$ on $\partial\B(0,R)$.

\begin{proposition}
\label{prop:excursion_length_starting_point}
There exist constants $c_1,c_2 > 0$ such that for all $u,v \in \partial\B(0,R)$ we have
\[  \left|\frac{\estart{\tau}{u}}{\estart{\tau}{v}} -1\right| \leq c_1\left(\frac{R}{n}\right)^{c_2}.\]
\end{proposition}

We prove the above proposition after establishing the following two lemmas.

\begin{lemma}
\label{lem:exit_distribution_intermediate}
There exists a constant $C > 1$ such that the following is true.  Suppose that $Q_1 < Q_2$ with $Q_1 \geq 2r$ and $Q \geq Q_2 \geq 2 Q_1$.  Let~$E_{r,Q} = \{\tau_{\partial\B(0,Q)} <\tau_{\partial\B(0,r)}\}$ and~$\sigma=\min\{t\geq 0: X(t) \in \partial\B(0,Q_2)\}$. Then
\[ \frac{1}{C} \leq \frac{\prcond{X(\sigma) = w}{E_{r,Q}}{u}}{\prcond{X(\sigma) = w}{E_{r,Q}}{v}} \leq C \quad\text{for all}\quad u,v \in \partial \B(0,Q_1) \quad\text{and}\quad w \in \partial \B(0,Q_2).\]
\end{lemma}

\begin{proof}
Note that the functions 
\begin{align*}
u \mapsto f(u) := \prstart{X(\sigma) =w, E_{r,Q}}{u} \quad\text{and}\quad u \mapsto g(u) := \prstart{E_{r,Q}}{u}
\end{align*}
are harmonic in $\B(0,Q_2) \setminus \B(0,r)$.  Consequently, it follows from Harnack's inequality (Lemma~\ref{lem:harnackineq}) that there exists a constant $C_1 \geq 1$ such that
\[ \frac{1}{C_1} \leq \frac{h(u)}{h(v)} \leq C_1 \quad\text{for}\quad h=f,g \quad\text{and all} \quad u,v \in \partial \B(0,Q_1) \quad\text{and}\quad w \in \partial \B(0,Q_2).\]
Since we have
\[ \frac{\prcond{X(\sigma) = w}{E_{r,Q}}{u}}{\prcond{X(\sigma) = w}{E_{r,Q}}{v}} = \frac{f(u)/g(u)}{f(v)/g(v)} = \frac{f(u)}{f(v)} \cdot \frac{g(v)}{g(u)}.\]
by taking $C = C_1^2$ proves the statement of the lemma.
\end{proof}

\begin{lemma}
\label{lem::exit_distribution}

Let $E_{r,Q}$ be as in Lemma~\ref{lem:exit_distribution_intermediate}, where $Q=2^kR$ for some $k$ and let $\sigma$ be the first time that $X$ hits $\partial\B(0,Q)$. There exist constants $c_1,c_2 > 0$ such that the following is true:
\[ \left| \frac{\prcond{ X(\sigma) = w}{E_{r,Q}}{u}}{\prcond{X(\sigma) = w}{E_{r,Q}}{v}} - 1 \right| \leq c_1 \left(\frac{R}{Q} \right)^{c_2} \quad\text{for all}\quad u,v \in \partial \B(0,R) \quad\text{and}\quad w \in \partial \B(0,Q).\]
\end{lemma}

\begin{proof}

For each $1 \leq j \leq k$, we let $Q_j = 2^j R$.  Note that $Q_k = Q$.  Lemma~\ref{lem::exit_distribution} implies that there exists a constant $\rho_0 > 0$ such that if $u,\til{u} \in \partial \B(0,Q_{j-1})$ and $Y,\til{Y}$ are random walks starting from $u,v$ respectively both conditioned on the event $E_{r,Q}$ and $\sigma_j,\til{\sigma}_j$ denotes the first time that they hit $\partial \B(0,Q_j)$ then
\[ \pr{ Y_{\sigma_j} = \til{Y}_{\til{\sigma}_j}} \geq \rho_0.\]
Let $\sigma_{k-1}$ be the first time that $X$ hits $\partial \B(0,Q_{k-1})$.  By iterating this, it follows that there exists a constant $\rho_1 \in (0,1)$ such that for all $u,v \in \partial \B(0,R)$ we have that
\begin{equation}
\label{eqn::k1_bound}
 \sum_{z \in \partial \B_{k-1}} \left| \prcond{ X(\sigma_{k-1}) = z}{E_{r,Q}}{u} - \prcond{ X(\sigma_{k-1}) = z}{E_{r,Q}}{v} \right| \leq \rho_1^{(k-1)}.
\end{equation}
Let $\sigma$ be the first time that $X$ hits $\partial \B(0,Q)$.  Then it follows that
\begin{align*}
  & \left| \frac{\prcond{ X(\sigma) = w}{E_{r,Q}}{u}}{\prcond{X(\sigma) = w}{E_{r,Q}}{v}} - 1 \right|
= \left| \frac{\prcond{X(\sigma) = w}{ E_{r,Q}}{u}-\prcond{X(\sigma) = w}{E_{r,Q}}{v}}{\prcond{X(\sigma) = w}{E_{r,Q}}{v}}  \right|\\
\leq&  \sum_{z \in \partial \B_{k-1}} \left| \prcond{ X(\sigma_{k-1}) = z}{E_{r,Q}}{u} - \prcond{ X(\sigma_{k-1}) = z}{E_{r,Q}}{v} \right| \frac{\prcond{ X(\sigma) = w}{E_{r,Q}}{z}}{\prcond{X(\sigma) = w}{E_{r,Q}}{v}}.
\end{align*}
By the strong Markov property, we note that
\[ \prcond{X(\sigma) = w}{E_{r,Q}}{v} \geq \min_{b \in \partial \B(0,Q_{k-1})} \prcond{ X(\sigma) = w}{E_{r,Q}}{b}.\]
Combining this with \eqref{eqn::k1_bound} and using Lemma~\ref{lem:exit_distribution_intermediate} we see that the above is bounded from above by
\[  \max_{a,b \in \partial \B(0,Q_{k-1})} \frac{\prcond{X(\sigma) = w}{E_{r,Q}}{a}}{\prcond{X(\sigma) = w}{E_{r,Q}}{b}} \times \rho_1^{k-1} \leq C \rho_1^{k-1},\]
and this finishes the proof.
\end{proof}

\begin{proof}[Proof of Proposition~\ref{prop:excursion_length_starting_point}]

Fix $y \in \partial \B(0,R)$.  Let $\xi$ be the length of time it takes for the random walk, after hitting $\partial \B(0,Q)$ where $Q=n/2$, to come hit $\partial \B(0,r)$, and then hit $\partial \B(0,R)$.  Then for $y \in \partial \B(0,R)$, we have that
\begin{align*}
   \estart{ \xi \1(E_{r,Q})}{y} \leq \estart{\tau}{y} \leq \estart{\tau_{\B(0,R)^c}}{y} + \estart{\xi \1(E_{r,Q})}{y}.
\end{align*}
Since in each round of the mixing time, the random walk has a positive chance of being outside of $\B(0,Q)$, it follows that there exists a constant $C > 0$ such that
\[ \estart{ \xi \1(E_{r,Q})}{y} \leq \estart{\tau}{y} \leq C n^2 + \estart{\xi \1(E_{r,Q})}{y}.\]
Let $\sigma$ be the first time that $X$ hits $\partial \B(0,Q)$.  We have that,
\begin{align*}
   &\estart{\xi \1(E_{r,Q})}{y}
        = \sum_{w \in \partial \B_Q} \estart{\xi}{w} \prcond{X(\sigma) = w}{E_{r,Q}}{y}\\
      &= \left(1+ \rho_1^{-(k-1)} \right) \sum_{w \in \partial \B_Q}  \estart{\xi}{w} \prcond{X(\sigma) = w}{E_{r,Q}}{z} \quad\text{(Lemma~\ref{lem::exit_distribution})}\\
     &= \left(1+ \rho_1^{-(k-1)} \right) \estart{\xi \1(E_{r,Q})}{z}.
\end{align*}

Combining, we have thus shown so far that
\[ \left(1- \rho_1^{-(k-1)} \right) \estart{\xi \1(E_{r,Q})}{z} \leq \estart{\tau}{y} \leq C n^2 +  \left(1+ \rho_1^{-(k-1)} \right)  \estart{\xi \1(E_{r,Q})}{z}.\]

The result then follows because $\estart{ \xi \1(E_{r,Q})}{z} \asymp n^d / r^{d-2}$ from Lemma~\ref{lem:trRlem}.
\end{proof}

\begin{proof}[Proof of Lemma~\ref{lem:int}]

Let $N$ be the index of the first excursion from $\partial \B(x,R)$ back to itself through $\partial \B(x,r)$ which hits $x$.  Then from Lemma~\ref{lem:hittingprob} it follows that $N$ is essentially a geometric random variable with expectation
\[
\frac{r^{d-2}}{\crw} \left( 1 + O\left(\frac{r}{R} \right) + O\left( \frac{1}{r^2} \right) \right) =\frac{r^{d-2}}{\crw}  (1+o(1)),\]
since $r=o(R)$. 
Let $\zeta_i$ be the length of the $i$-th such excursion.  If~$z \in \partial \B(x,R)$, then we have that
\begin{align}\label{eq:thittrr}
 \estart{\tau_x}{z} &= \sum_{i=1}^\infty \escond{ \zeta_i}{N \geq i}{z}\pr{N \geq i} -
\estart{\tau_{\partial\B(0,R)}}{x} 
\\&= \sum_{i=1}^\infty \escond{ \zeta_i}{N \geq i}{z}\pr{N \geq i} + O(n^2).
\end{align}
Proposition~\ref{prop:excursion_length_starting_point} gives that
\begin{align*}
\escond{\zeta_i}{ N \geq i}{z} = \trb_{r,R}(1+o(1)),
\end{align*}
and hence putting everything together we obtain
\begin{align*}
\estart{\tau_x}{z} = \frac{\trb_{r,R}}{\prstart{\tau_x<\tau_{\partial\B(x,R)}}{\pi} }(1+o(1)).
\end{align*}
If $z \notin \B(x,R)$, then
\[ \estart{\tau_x}{z} = \estart{\tau_{\partial\B(x,R)}}{z}+ \frac{\trb_{r,R}}{\prstart{\tau_x<\tau_{\partial\B(x,R)}}{\pi}}(1+o(1)).\]
In this case, by Lemma~\ref{lem:trRlem} we have $\estart{\tau_{\partial\B(x,R)}}{z}=O(n^d/ R^{d-2})$, and hence from the above we get that if $z\notin \B(x,R)$, then
\[
\estart{\tau_x}{z} = \frac{\trb_{r,R}}{\prstart{\tau_x<\tau_{\partial\B(x,R)}}{\pi}}(1+o(1)).
\]  
If~$z \in \B(x,R)$, then we have that
\[ \estart{\tau_x}{z} \leq \estart{\tau_{\partial\B(x,R)}}{z} + \frac{\trb_{r,R}}{\prstart{\tau_x<\tau_{\partial\B(x,R)}}{\pi}}(1+o(1)).
\]
Therefore, combining everything we get that
\[ \tmax= \frac{\trb_{r,R}}{\prstart{\tau_x<\tau_{\partial\B(x,R)}}{\pi}}(1+o(1))\]
and this concludes the proof.
\end{proof}

\section*{Acknowledgments}

We thank Amir Dembo, Roberto Imbuzeiro-Oliveira, Yuval Peres, and Augusto Teixeira for helpful discussions.

\makeatletter
\def\@rst #1 #2other{#1}
\renewcommand\MR[1]{\relax\ifhmode\unskip\spacefactor3000 \space\fi
  \MRhref{\expandafter\@rst #1 other}{#1}}
\newcommand{\MRhref}[2]{\href{http://www.ams.org/mathscinet-getitem?mr=#1}{MR#2}}
\makeatother

\phantomsection
\bibliographystyle{hmralpha}
\bibliography{biblio}

\begin{thebibliography}{DDMP13}

\bibitem[Ald91]{ALD_thresh}
David~J. Aldous.
\newblock Threshold limits for cover times.
\newblock {\em J. Theoret. Probab.}, 4(1):197--211, 1991. \MR{1088401
  (91m:60123)}

\bibitem[Bel12]{BEL}
David Belius.
\newblock Cover levels and random interlacements.
\newblock {\em Ann. Appl. Probab.}, 22(2):522--540, 2012. \MR{2953562}

\bibitem[BH91]{BH}
M.~J. A.~M. Brummelhuis and H.~J. Hilhorst.
\newblock Covering of a finite lattice by a random walk.
\newblock {\em Phys. A}, 176(3):387--408, 1991. \MR{1130067 (92m:82058)}

\bibitem[DDMP13]{DDMP}
Amir Dembo, Jian Ding, Jason Miller, and Yuval Peres.
\newblock Cut-off for lamplighter chains on tori: dimension interpolation and
  phase transition.
\newblock 2013.

\bibitem[DPR03]{DPRZ_bm_manifold}
Amir Dembo, Yuval Peres, and Jay Rosen.
\newblock Brownian motion on compact manifolds: cover time and late points.
\newblock {\em Electron. J. Probab.}, 8:no. 15, 14, 2003. \MR{1998762
  (2004g:58047)}

\bibitem[DPRZ01]{DPRZ_thick}
Amir Dembo, Yuval Peres, Jay Rosen, and Ofer Zeitouni.
\newblock Thick points for planar {B}rownian motion and the {E}rdos-{T}aylor
  conjecture on random walk.
\newblock {\em Acta Math.}, 186(2):239--270, 2001. \MR{1846031 (2002k:60106)}

\bibitem[DPRZ04]{DPRZ_cov}
Amir Dembo, Yuval Peres, Jay Rosen, and Ofer Zeitouni.
\newblock Cover times for {B}rownian motion and random walks in two dimensions.
\newblock {\em Ann. of Math. (2)}, 160(2):433--464, 2004. \MR{2123929
  (2005k:60261)}

\bibitem[DPRZ06]{DPRZ_late}
Amir Dembo, Yuval Peres, Jay Rosen, and Ofer Zeitouni.
\newblock Late points for random walks in two dimensions.
\newblock {\em Ann. Probab.}, 34(1):219--263, 2006. \MR{2206347 (2007b:60110)}

\bibitem[Dur10]{Durrett_prob}
Rick Durrett.
\newblock {\em Probability: theory and examples}.
\newblock Cambridge Series in Statistical and Probabilistic Mathematics.
  Cambridge University Press, Cambridge, fourth edition, 2010. \MR{2722836
  (2011e:60001)}

\bibitem[FP99]{FP_kac}
P.~J. Fitzsimmons and Jim Pitman.
\newblock Kac's moment formula and the {F}eynman-{K}ac formula for additive
  functionals of a {M}arkov process.
\newblock {\em Stochastic Process. Appl.}, 79(1):117--134, 1999. \MR{1670526
  (2000a:60136)}

\bibitem[{Imb}11]{IMBUZ_hit}
R.~{Imbuzeiro Oliveira}.
\newblock {Mean field conditions for coalescing random walks}.
\newblock {\em ArXiv e-prints}, September 2011, 1109.5684.

\bibitem[IP13]{ImbuzPrata}
R.~{Imbuzeiro Oliveira} and Alan Prata.
\newblock Late points and cover times for locally transient random walks.
\newblock 2013.

\bibitem[LL10]{LawLim}
Gregory~F. Lawler and Vlada Limic.
\newblock {\em Random walk: a modern introduction}, volume 123 of {\em
  Cambridge Studies in Advanced Mathematics}.
\newblock Cambridge University Press, Cambridge, 2010. \MR{2677157
  (2012a:60132)}

\bibitem[LPW09]{LPW}
David~A. Levin, Yuval Peres, and Elizabeth~L. Wilmer.
\newblock {\em Markov chains and mixing times}.
\newblock American Mathematical Society, Providence, RI, 2009.
\newblock With a chapter by James G. Propp and David B. Wilson. \MR{2466937
  (2010c:60209)}

\bibitem[Mat88]{MATT_cover}
Peter Matthews.
\newblock Covering problems for {B}rownian motion on spheres.
\newblock {\em Ann. Probab.}, 16(1):189--199, 1988. \MR{920264 (89a:60190)}

\bibitem[MP12]{MP_unif}
Jason Miller and Yuval Peres.
\newblock Uniformity of the uncovered set of random walk and cutoff for
  lamplighter chains.
\newblock {\em Ann. Probab.}, 40(2):535--577, 2012. \MR{2952084}

\bibitem[PR04]{PR_lamp}
Yuval Peres and David Revelle.
\newblock Mixing times for random walks on finite lamplighter groups.
\newblock {\em Electron. J. Probab.}, 9:no. 26, 825--845, 2004. \MR{2110019
  (2005m:60007)}

\bibitem[Pra12]{Prata_thesis}
Alan Prata.
\newblock PhD thesis, 2012.

\bibitem[Spi64]{Spitzer}
Frank Spitzer.
\newblock Electrostatic capacity, heat flow, and {B}rownian motion.
\newblock {\em Z. Wahrscheinlichkeitstheorie und Verw. Gebiete}, 3:110--121,
  1964. \MR{0172343 (30 \#2562)}

\bibitem[Szn10]{SZNIT_interlacements}
Alain-Sol Sznitman.
\newblock Vacant set of random interlacements and percolation.
\newblock {\em Ann. of Math. (2)}, 171(3):2039--2087, 2010. \MR{2680403
  (2011g:60185)}

\end{thebibliography}

\end{document}